\documentclass{article}

\usepackage[utf8]{inputenc} 
\usepackage[T1]{fontenc}    
\usepackage[a4paper, total={6in, 8in}]{geometry}
\usepackage[colorlinks=true]{hyperref}       
\usepackage{natbib}
\usepackage{authblk}
\usepackage{url}            
\usepackage{booktabs}       
\usepackage{amsfonts}       
\usepackage{nicefrac}       
\usepackage{microtype}      
\usepackage[ruled,vlined,linesnumbered,noresetcount]{algorithm2e}
\usepackage{amsthm}
\usepackage{float}
\usepackage{subcaption}
\usepackage{parskip}

\usepackage{enumitem}
\setlist{nosep,leftmargin=*}
\usepackage{todonotes}

\usepackage{mdl}
\usepackage{autonum}

\title{Locally Differentially Private Minimum Finding}

\author[1]{Kazuto Fukuchi\thanks{kazuto.fukuchi@riken.jp}}
\author[2]{Chia-Mu Yu\thanks{chiamuyu@gmail.com}}
\author[3]{Arashi Haishima\thanks{arashi@mdl.cs.tsukuba.ac.jp}}
\author[3,1]{Jun Sakuma\thanks{jun@cs.tsukuba.ac.jp}}
\affil[1]{RIKEN AIP, Japan}
\affil[2]{National Chung Hsing University, Taiwan}
\affil[3]{University of Tsukuba, Japan}

\begin{document}

\maketitle

\begin{abstract}
  \noindent We investigate a problem of finding the minimum, in which each user has a real value and we want to estimate the minimum of these values under the local differential privacy constraint. We reveal that this problem is fundamentally difficult, and we cannot construct a mechanism that is consistent in the worst case. Instead of considering the worst case, we aim to construct a private mechanism whose error rate is adaptive to the easiness of estimation of the minimum. As a measure of easiness, we introduce a parameter $\alpha$ that characterizes the fatness of the minimum-side tail of the user data distribution. As a result, we reveal that the mechanism can achieve $O((\ln^6N/\epsilon^2N)^{1/2\alpha})$ error without knowledge of $\alpha$ and the error rate is near-optimal in the sense that any mechanism incurs $\Omega((1/\epsilon^2N)^{1/2\alpha})$ error. Furthermore, we demonstrate that our mechanism outperforms a naive mechanism by empirical evaluations on synthetic datasets. Also, we conducted experiments on the MovieLens dataset and a purchase history dataset and demonstrate that our algorithm achieves $\tilde{O}((1/N)^{1/2\alpha})$ error adaptively to $\alpha$. 
\end{abstract}

\section{Introduction}

Statistical analyses with individuals' data have a significant benefit to our social lives. However, using individuals' data raises a serious concern about privacy, and privacy preservation is therefore increasingly demanding by social communities. For example, the European Commission~(EC) approved a new regulation regarding data protection and privacy, the General Data Protection Regulation~(GDPR), which has been in effect since May 2018. With this regulation, any service provider in the world must follow GDPR when providing services to any individuals in the EU. 

Motivated by the privacy concern, many researchers developed statistical analysis methods with a guarantee of {\em Differential privacy}~\citep{Dwork2006}. The differential privacy prevents privacy leakage in the central model in which a trusted central server\footnote{The terms \emph{server} and \emph{aggregator} are used interchangeably throughout paper.} gathers the individuals' data and then publishes some statistical information about the gathered data to an untrusted analysist. One limitation of this model is that it requires a trusted central server that processes a differentially private algorithm. 

A notion of {\em local differential privacy}~(LDP) was introduced for preventing privacy leakage to the {\em untrusted} central server. \citet{Warner1965RandomizedBias} firstly introduced it in the context of surveying individuals. Afterward, \citet{Evfimievski2003} provided a definition of local privacy in a general situation. Many researchers proposed some statistical analysis methods with a guarantee of the local differential privacy. For example, methods for mean and median estimation~\citep{Duchi2013LocalRates}, distribution estimation~\citep{Erlingsson2013ccs,Fanti2016popet,Ren2018tifs}, and heavy hitter estimation~\citep{Bassily2015stoc} under the LDP guarantee have been investigated so far.

In this paper, we deal with the {\em minimum finding problem} under the local differential privacy constraint. In this problem, a number of users hold a real value individually, which can be a sensitive value, and an analyst is interested in the minimum among the values. The minimum finding problem is a primitive but fundamental component for statistical analysis. Even under the privacy constraint, the minimum finding is a necessary first step of statistical analyses. 

As we describe later, our mechanism employs binary search to find the interval that contains the minimum. Binary search with local differential privacy has been employed in \citet{Gaboardi2019} for the first time as a component of a mechanism to estimate mean with a confidence interval. Naive application of this mechanism enables minimum finding with local differential privacy, whereas a straightforward application of their utility analysis to minimum finding does not necessarily result in consistent estimation. This is because their utility analysis is specifically derived for their main task, mean estimation with a confidence interval. Further analysis with an additional assumption is needed for deriving a locally private mechanism that can consistently estimate the minimum.

Our contributions are listed as follows:
\begin{description}
 \item[Hardness in the worst case] We reveal that the minimum finding problem under the local differential privacy constraint is fundamentally difficult in the worst case. We will prove the fact that there is no locally differentially private mechanism that consistently estimates the minimum under the worst case users' data distribution.
 \item[LDP mechanism with adaptiveness to $\alpha$-fatness] Instead of considering the worst case, we construct a locally differentially private mechanism that is {\em adaptive to the easiness} of estimation of the minimum, which is determined by the underlying user data distribution. As a measure of easiness, we introduce $\alpha$-fatness which characterizes the fatness of the minimum-side tail of the user data distribution. Here, a smaller $\alpha$ indicates that the tail is fatter. The minimum finding problem becomes apparently easier when the underlying distribution is fat because we can expect that a greater portion of data is concentrated around the minimum if the distribution is fatter. Hence, we can expect that the decreasing rate of the estimation error becomes smaller as $\alpha$ decreases. The definition of $\alpha$-fatness is given as follows: 
\begin{definition}[$\alpha$-fatness]\label{def:fatness}
 Let $F$ be the cumulative function of the user data distribution supported on $[-1,1]$. For a positive real $\alpha$, the distribution of $F$ is $\alpha$-fat if there exist universal constants $C > 0$ and $\bar{x} \in [-1,1]$ such that for all $x_{\min} < x < \bar{x}$, $F(x) \ge C\paren*{x - x_{\min}}^\alpha$ where $x_{\min} = \inf\cbrace{x \in [-1,1] : F(x) > 0}$ is the minimum under $F$. 
\end{definition}
For example, any truncated distribution, such as the truncated normal distribution, satisfies \cref{def:fatness} with $\alpha = 1$. The beta distribution with parameters $\alpha$ and $\beta$ is $\alpha$-fat. For simplicity, we say $F$ is $\alpha$-fat if the $F$'s distribution is $\alpha$-fat. 
 \item[Utility analyses] We derive adaptive upper bounds on the mean absolute error of the present mechanism as utility analyses and reveal that these bounds are nearly tight. Under the assumption that the server knows a lower bound on $\alpha$, we show that the mean absolute error is $O\paren{\paren{\nicefrac{\ln^3N}{\epsilon^2N}}^{1/2\alpha}}$, where $N$ denotes the number of users, and $\epsilon$ is the privacy parameter. If $\alpha$ is unknown to the server, we show that the mean absolute error is $O\paren{\paren{\nicefrac{\ln^6N}{\epsilon^2N}}^{1/2\alpha}}$. Also, we prove that these upper bounds are nearly tight in the sense that any locally differentially private mechanism incurs $\Omega\paren{\paren{\nicefrac{1}{\epsilon^2N}}^{1/2\alpha}}$ error under the $\alpha$-fatness assumption. The error rates of our mechanism become slower as $\alpha$ increases; this reflects the intuition about the easiness of estimation mentioned before. Note that this decreasing rate can be achieved even though the algorithm can use only imperfect knowledge on $\alpha$ ~(e.g., lower bound on $\alpha$) or no information about $\alpha$. 
 
 \item[Empirical evaluation] We conducted some experiments on real and synthetic datasets for evaluating the performance of the proposed mechanism. In the experiment on the synthetic datasets, we firstly confirm the tightness of the theoretical bounds regarding $N$ and $\epsilon$. Furthermore, we demonstrate by the experiment that the present mechanism outperforms a baseline method based on the Laplace mechanism. In the experiment on the real datasets, we evaluate the performance of the proposed mechanism on the MovieLens dataset and a customers’ purchase history dataset. As a result, we present that the proposed mechanism succeeds to achieve $\tilde{O}(1/N^{1/2\alpha})$ rate adaptively to $\alpha$, where the notation $\tilde{O}$ ignores the logarithmic factor.
\end{description}

All the missing proofs can be found in the appendix.

\section{Preliminaries}\label{sec:preliminaries}

We introduce the interactive version of the privacy definition given by \citet{Duchi2013LocalRates}. Suppose that an individual has a data on a domain $\dom{X}$. To preserve privacy, the data is converted by a random mechanism $\dom{M}$ before sending the data to the untrusted server, where the domain of the converted value is denoted as $\dom{Z}$. In the interactive setting, a mechanism $\mathcal{M}$ is a probabilistic mapping from $\dom{X}\times\dom{Z}^{*}$ to $\dom{Z}$, where $\dom{Z}^* = \bigcup_{n=1}^\infty\dom{Z}^n$. This means that when the mechanism converts the user's data, it can utilize the sanitized data that have been already revealed to the server. Privacy of the mechanism $\dom{M}$ is defined as follows: 
\begin{definition}[Local differential privacy~\citep{Duchi2013LocalRates}]\label{def:local-dp}
 A stochastic mechanism $\dom{M}$ is $\epsilon$-locally differentially private if for all $x,x' \in \dom{X}$, all $z \in \dom{Z}^*$, and all $S \in \sigma(\dom{Z})$,
 \begin{align}
     \p\cbrace*{\dom{M}(x,z) \in S} \le e^\epsilon\p\cbrace*{\dom{M}(x',z) \in S},
 \end{align}
 where $\sigma(\dom{Z})$ denotes an appropriate $\sigma$-field of $\dom{Z}$.
\end{definition}
The parameter $\epsilon$ determines a level of privacy; that is, smaller $\epsilon$ indicates stronger privacy protection. Roughly speaking, the local differential privacy guarantees that the individual's data cannot be certainly inferred from the privatized data even if an adversary has unbounded computational resources and any prior knowledge. 

As a simple implementation of the locally differentially private mechanism, the randomized response proposed by \citet{Warner1965RandomizedBias} is known. This is a mechanism for binary data and outputs binary value. Let $\dom{X}=\dom{Z}=\cbrace{-1,1}$, and let $x$ and $z$ be the individual's data and privatized data by the randomized response, respectively. Then, the randomized response flips the individual's data $x$ with probability $\nicefrac{1}{1+e^\epsilon}$, and thus we have $z = x$ with probability $\nicefrac{e^\epsilon}{1+e^\epsilon}$ and $z = -x$ with probability $\nicefrac{1}{1+e^\epsilon}$. This mechanism ensures $\epsilon$-local differential privacy.

{\bfseries Notations.}
We denote the indicator function as $\ind{x}$ for an predicate $x$. Let $\sign(x) = 1$ if $x \ge 0$, and $\sign(x) = -1$ if $x < 0$. For an event $\event$, we denote its complement as $\event^c$.
\section{Problem Setup}

Suppose there are two stakeholders; {\em users} and {\em aggregator}. There are $N$ users. They have real-valued data $x_i \in [-1,1]$ and want $x_i$ to be private. Let $x_{(1)} \le x_{(2)} \le ... \le x_{(N)}$ be ordered data. We investigate two similar settings regarding users' data generation process.
\newcommand\boldparen[1]{\normalfont({\bfseries #1})}
\begin{description}[format=\boldparen]
 \item[Fixed data] The users' data are fixed by some unknown rule. 
 \item[i.i.d. data] The users' data are drawn i.i.d. from some unknown distribution. 
\end{description}
The aggregator in the fixed data setting aims to obtain the minimum among the users' data, whereas he/she in the i.i.d. data setting aims to obtain the minimum within the support of the underlying users' data distribution.

The unknown rule or distribution is described by a non-decreasing function $F:[0,1]\to[-1,1]$. In the fixed data setting, the function $F$ determines the empirical cumulative distribution of the users' data. More precisely, the users' data are determined such that $F(x_{(i)}) = (i-1)/(N-1)$ for all $i=1,...,N$. In the i.i.d. data setting, $F$ is the cumulative distribution function of the unknown user data distribution. In the both settings, the minimum of the users' data is defined as $x_{\min} = \inf\cbrace{x : F(x) > 0}$. 

The utility of the estimation is measured by the absolute mean error between the true and estimated values of the minimum. Let $\tilde{x}$ be the estimated value. Then, the absolute mean error is defined as
\begin{align}
  \mathrm{Err} = \Mean\bracket*{\abs*{x_{\min} - \tilde{x}}}, \label{eq:err}
\end{align}
where the expectation is taken over randomness of the sanitization mechanism and data generation~(in the i.i.d. data setting). When the users send information regarding $x_i$ to the aggregator, it must be sanitized in the locally differentially private way. Given the privacy parameter $\epsilon > 0$, our goal is to estimate $x_{\min}$ that minimizes the absolute mean error in \cref{eq:err} under the constraint of the $\epsilon$-local differential privacy.

In later discussions, we use $\tilde{F}(x) = \frac{1}{N}\sum_{i=1}^N\ind{x_i \le x}$. We define the quantile function of $F$ and $\tilde{F}$ as $F^*(\gamma) = \inf\cbrace{\tau : F(\tau) \ge \gamma}$ and $\tilde{F}^*(\gamma) = \inf\cbrace{\tau : \tilde{F}(\tau) \ge \gamma}$, respectively. 

\section{Algorithm}
\begin{figure}[tbhp]
\begin{minipage}[tbhp]{.48\textwidth}
\begin{algorithm}[H]
  \SetKwInOut{Input}{Input}
  \Input{Depth $L$}
  Initialize $\ell_1 = -1$ and $r_1 = 1$\;
  \For{$t = 1$ to $L$}{
    $\tau_t = \frac{\ell_t + r_t}{2}$ \;
    Each user reports $z_i = \sign(\tau_t - x_i)$ \vspace{3.6em}\label{line:access}\;
    The aggregator obtains $z= (z_1,...,z_N)$ \label{line:obtain} \;
    Calculate $\Phi(z) = \frac{1}{2N}\sum_{i=1}^Nz_i + \frac{1}{2}$ \vspace{1.4em}\label{line:phi} \;
    \If{$\Phi(z) > 0$}{ \label{line:if}
      $\ell_{t+1} = \ell_t$ and $r_{t+1} = \tau_t$ \label{line:then}
    }\Else{ \label{line:else}
      $\ell_{t+1} = \tau_t$ and $r_{t+1} = r_t$ \label{line:else-then}
    }
  }
  \Return{$\tilde{x} = \frac{\ell_{L+1} + r_{L+1}}{2}$}
  \caption{Non-private finding minimum}\label{alg:no-private-find-min}
\end{algorithm}
\end{minipage}
\hspace{1em}
\begin{minipage}[tbhp]{.48\textwidth}
\begin{algorithm}[H]
  \setcounter{AlgoLine}{0}
  \SetKwInOut{Input}{Input}
  \Input{Depth $L$ and a threshold $\gamma$}
  Initialize $\ell_1 = -1$ and $r_1 = 1$ \;
  \For{$t = 1$ to $L$}{
    $\tau_t = \frac{\ell_t + r_t}{2}$ \;
    Each user reports $z'_i$ obtained by sanitizing $\sign(\tau_t - x_i)$ via randomized response with the privacy parameter $\epsilon/L$ \label{line:private-access} \;
    The aggregator obtains $z' = (z'_1,...,z'_N)$ \;
    Calculate $\Phi'(z')\!=\!\frac{1}{2N}\frac{e^{\epsilon/L}+1}{e^{\epsilon/L}-1}\sum_{i=1}^Nz'_i + \frac{1}{2}$ \label{line:private-phi} \;
    \If{$\Phi'(z') \ge \gamma$}{ \label{line:private-if}
      $\ell_{t+1} = \ell_t$ and $r_{t+1} = \tau_t$
    }\Else{
      $\ell_{t+1} = \tau_t$ and $r_{t+1} = r_t$
    }
  }
  \Return{$\tilde{x} = \frac{\ell_{L+1} + r_{L+1}}{2}$}
  \caption{Locally private finding minimum}\label{alg:find-min}
\end{algorithm}
\end{minipage}
\end{figure}

In this section, we derive an algorithm for the locally private finding minimum problem. \cref{alg:no-private-find-min} shows the non-private version of the proposed algorithm. It employs the binary search algorithm to find the interval containing the minimum from $2^L$ distinct intervals obtained by evenly dividing the data domain $[-1,1]$, where $L$ is some positive integer. More precisely, \cref{alg:no-private-find-min} iteratively updates the interval $[\ell_t,r_t]$, where the left-endpoint, midpoint, and right-endpoint of the interval are denoted as $\ell_t$, $\tau_t$, and $r_t$, respectively. In Line 1, \cref{alg:no-private-find-min} initializes the first interval $[\ell_1,r_1]$ as the data domain. Then, for each round $t$, the algorithm halves the interval into $[\ell_t,\tau_t)$ and $[\tau_t,r_t]$ and then chooses either of them that contains the minimum $x_{(1)}$~(in Lines 3-10). After $L$ iterations, the interval becomes the desired one. The algorithm outputs the middle of the interval as the estimated value~(in Line 11). Because the length of the last interval is $2^{-L+1}$ by construction, the error of the estimated value is up to $2^{-L}$.

To identify which $[\ell_t,\tau_t)$ and $[\tau_t,r_t]$ contains the minimum, \cref{alg:no-private-find-min} first asks each user whether or not his/her data is smaller than $\tau_t$~(in Line 4). After that, \cref{alg:no-private-find-min} calculates the empirical cumulative distribution at $\tau_t$, $\tilde{F}(\tau_t)$, based on their responses. In \cref{alg:no-private-find-min}, it is denoted as $\Phi(z)$ in Line 6. Then, $[\ell_t,\tau_t)$ contains the minimum if $\tilde{F}(\tau_t) > 0$, and $[\tau_t,r_t]$ does otherwise.

\cref{alg:find-min} shows the privatized version of \cref{alg:no-private-find-min}. \cref{alg:no-private-find-min} accesses the users' data only through a query that asks whether or not his/her data is smaller than $\tau_t$. We sanitize the query by using the randomized response described in \cref{sec:preliminaries} in Line 4 of \cref{alg:find-min}. Since the randomized response introduces a noise into the query's response, we modify Lines 6 and 7 of \cref{alg:no-private-find-min}. In Line 6, instead of calculating $\Phi(z)$, \cref{alg:find-min} calculates the unbiased estimated value of $\tilde{F}(\tau_t)$, which is denoted as $\Phi'(z')$. Unbiasedness of the estimated value can be confirmed by an elementary calculation. In Line 7, because $\Phi'(z')$ involves error due to sanitization, we introduce a threshold $\gamma$ instead of $0$. 

In \cref{alg:find-min}, there are two free parameters; $L$ and $\gamma$. We investigate an appropriate choice of $L$ and $\gamma$ by analyzing the absolute mean error of this algorithm. The results of the analyses are demonstrated in the next section. We remark that due to the binary search strategy in our proposed method, one can easily see that our proposed method can be easily adapted to maximum finding.

\section{Analyses}\label{sec:analysis}

{\bfseries Hardness in the worst case.} We first show that the private finding minimum problem is fundamentally difficult. Indeed, we cannot construct a locally differentially private algorithm that consistently estimates the minimum in the worst case users' data:
\begin{theorem}\label{thm:hard-worst}
 Suppose $\epsilon$ is fixed. In the both setting, for any $\epsilon$-locally differentially private mechanism, there exists $F$ such that $\mathrm{Err} = \Omega(1)$ with respect to $N$.
\end{theorem}
From the theorem, we can see that the finding minimum problem cannot be solved with a reasonable utility. In \cref{thm:hard-worst}, we consider a situation where the minimum point is isolated to all the other points; that is, $x_1 = -1$ and $x_i = 1$ for $i=2,...,N$. The worst case distribution is not $\alpha$-fat for any finite $\alpha$. 

{\bfseries Adaptive upper bounds and privacy of \cref{alg:find-min}.} Next, assuming $\alpha$-fatness of the user's distribution, we reveal the privacy guarantee and the dependency of the error on $\epsilon$ and $N$ regarding \cref{alg:find-min}. 
\begin{theorem}\label{thm:adjusted}
 For any choice of $\epsilon$, $L$, and $\gamma$, \cref{alg:find-min} is $\epsilon$-locally differentially private. Moreover, for some $\alpha > 0$, suppose $F$ is $\alpha$-fat. For a sequence $h_N$, let $\gamma = \sqrt{\nicefrac{4e^{\epsilon/L}(1+e^{\epsilon/L})h_N}{(e^{\epsilon/L}-1)^2N}}$. Then, in both of the fixed and i.i.d. data settings, if $\nicefrac{L^2h_N}{\epsilon^2N} = o(1)$, \cref{alg:find-min} incurs an error as
 \begin{align}
     \mathrm{Err} = O\paren*{\paren*{\nicefrac{L^2h_N}{\epsilon^2N}}^{1/2\alpha} + e^{-h_N} + 2^{-L}}.
 \end{align}
\end{theorem}
In \cref{thm:adjusted}, there are two free parameters, $h_N$ and $L$, which should be selected by the aggregator. We obtain $O((L^2h_N/\epsilon^2N)^{1/2\alpha})$ error rate by choosing $h_N$ and $L$ so that the second and third terms in \cref{thm:adjusted} are smaller than the first term.

Let us consider the case where the aggregator has a prior knowledge regarding a lower bound on $\alpha$. In this case, an appropriate choice of $h_N$ and $L$ is shown in the following corollary.
\begin{corollary}\label{cor:known-alpha}
 For some $\alpha > 0$, suppose $F$ is $\alpha$-fat. Let $h_N = \nicefrac{\ln(N)}{2\alpha}$ and $L = \Theta(\log_2 N)$ such that  $L \ge \log_2(N)/2\alpha$. Then, \cref{alg:find-min} incurs an error as
 \begin{align}
     \mathrm{Err} = O\paren*{\paren*{\nicefrac{\ln^{3}(N)}{\epsilon^2N}}^{1/2\alpha}}.
 \end{align}
\end{corollary}
The next corollary is useful if the aggregator does not have any prior information about $\alpha$. In this case, the decreasing rate of the error is slightly worse than \cref{cor:known-alpha}.
\begin{corollary}\label{cor:unknown-alpha}
 For some $\alpha > 0$, suppose $F$ is $\alpha$-fat. Let $h_N = \Theta(\log^2(N))$ and $L = \Theta(\log^2(N))$. Then, \cref{alg:find-min} incurs an error as
 \begin{align}
     \mathrm{Err} = O\paren*{\paren*{\nicefrac{\ln^{6}(N)}{\epsilon^2N}}^{1/2\alpha}}.
 \end{align}
\end{corollary}
As well as the intuition, the decreasing rate becomes faster as $\alpha$ decreases in both settings.

{\bfseries Lower bound for the locally private minimum finding.} For confirming tightness of \cref{cor:known-alpha,cor:unknown-alpha}, we derive a minimax lower bound for the locally private minimum finding. 
\begin{theorem}\label{thm:fat-worst}
 Fix $\epsilon \in [0,22/35]$, $\alpha$, and $C$. In the i.i.d. data setting, for any $\epsilon$-locally differentially private mechanism, there exists $F$ satisfies \cref{def:fatness} with $\alpha$ and $C$ such that for a increasing sequence of $N$ and a decreasing sequence of $\epsilon$,
 \begin{align}
    \mathrm{Err} = \Omega\paren*{\paren*{\nicefrac{1}{\epsilon^2N}}^{1/2\alpha}}.
 \end{align}
\end{theorem}
As proved in \cref{thm:fat-worst}, any locally private mechanism incurs $\Omega((\nicefrac{1}{\epsilon^2N})^{1/2\alpha})$ error which matches the upper bounds shown in \cref{cor:known-alpha,cor:unknown-alpha} up to log factors. Note that we derive the lower bound in \cref{thm:fat-worst} in a situation where the aggregator knows the fatness parameter $\alpha$. If the aggregator does not know $\alpha$, the minimax error might be greater than the one shown in \cref{thm:fat-worst}.

\section{Experiment}
Here, we present experimental results on synthetic, the MovieLens, and a purchase history datasets to show the accuracy advantage of our proposed method and confirm the correctness of the our theoretical analysis. 

\subsection{Synthetic Data}\label{sec:syn}

\begin{figure}[tbhp]
  \centering
  \includegraphics[width=.32\textwidth]{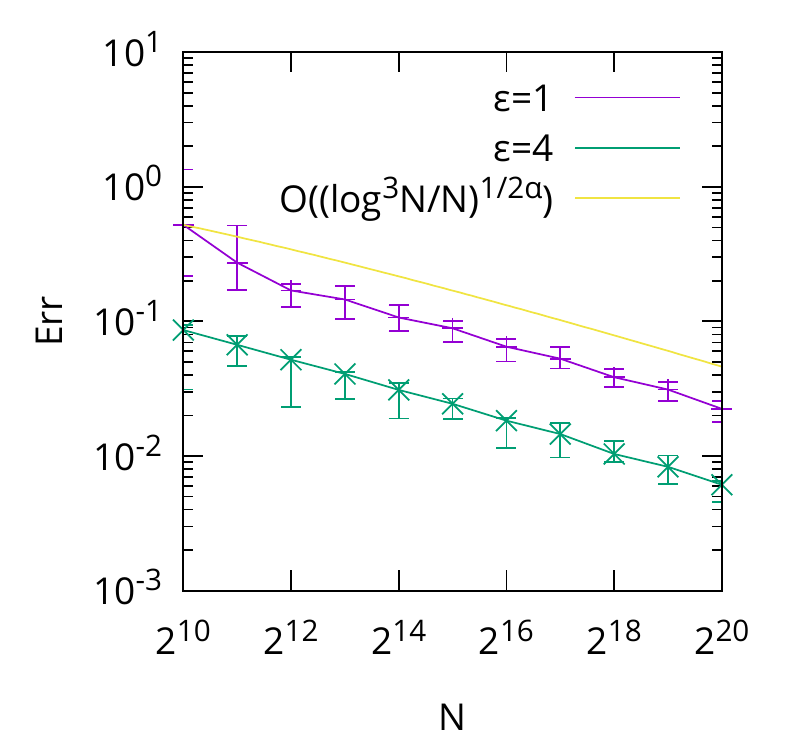}
  \includegraphics[width=.32\textwidth]{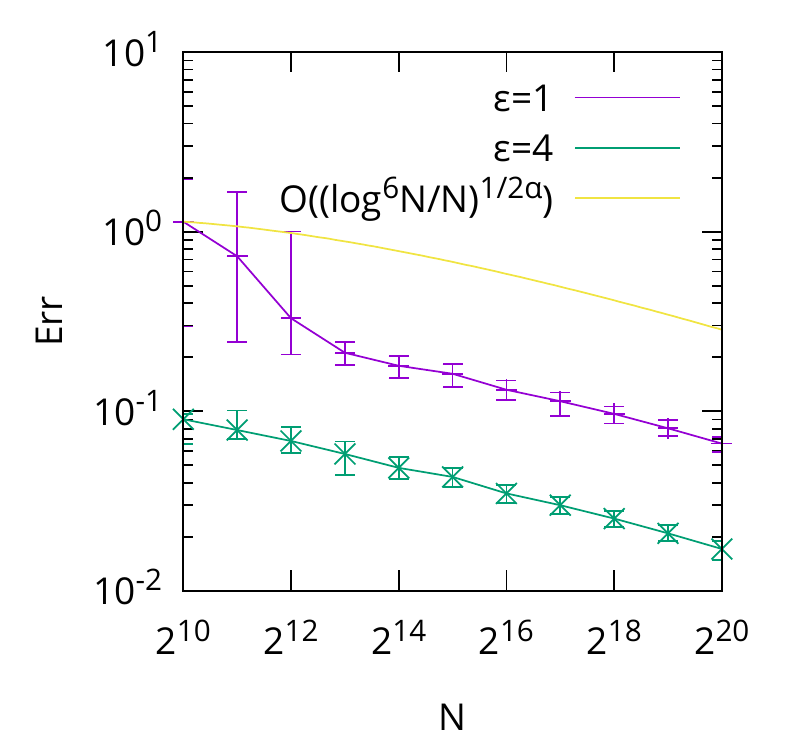}
  \includegraphics[width=.32\textwidth]{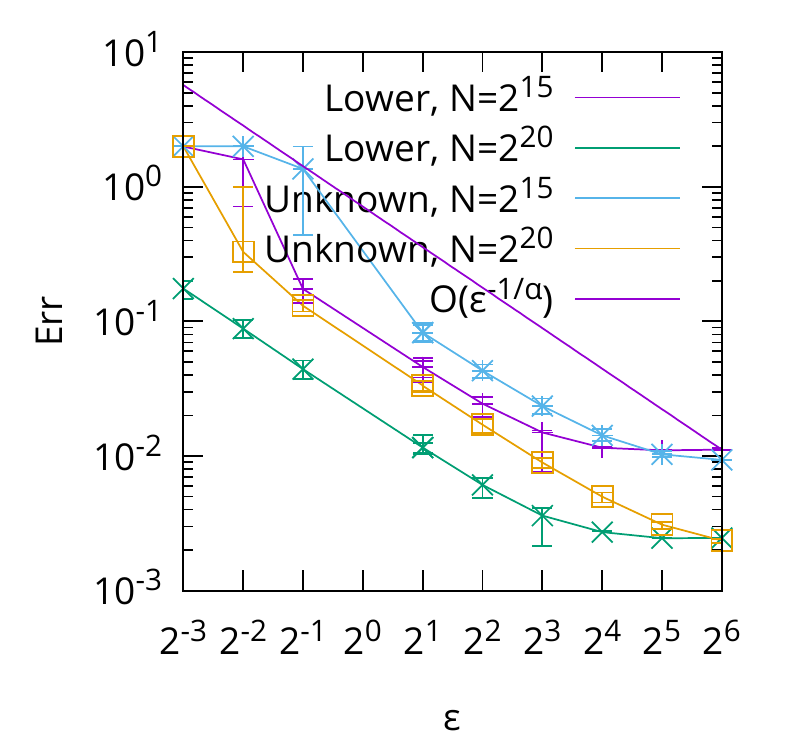}
  \caption{$\mathrm{Err}$ v.s. $N$~(left and middle) and $\mathrm{Err}$ v.s. $\epsilon$~(right). The left figure depicts the result with {\bfseries Known $\alpha$}, and the middle figure depicts the result with {\bfseries Unknown $\alpha$}. }\label{fig:err-vs-n-eps}
  \includegraphics[width=.32\textwidth]{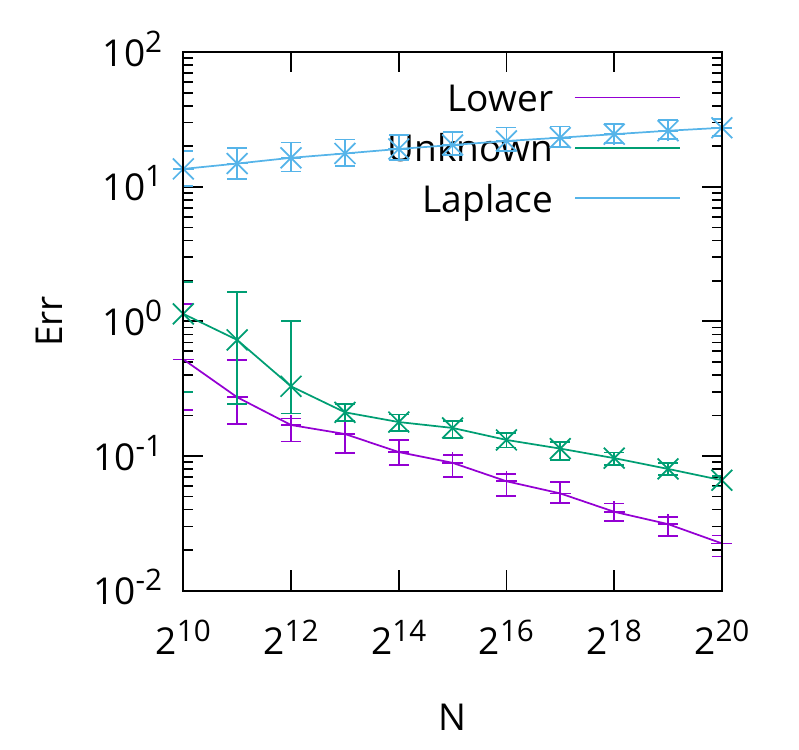}
  \includegraphics[width=.32\textwidth]{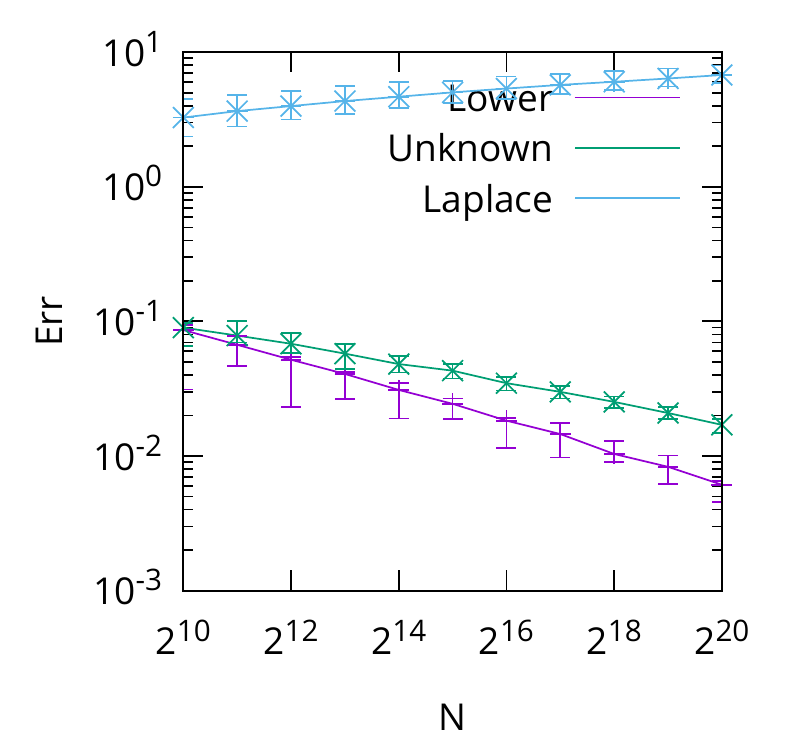}
  \caption{Comparison between our methods and the baseline method with $\epsilon=1$~(left) and $\epsilon=4$~(right).}\label{fig:cmp}
\end{figure}

We investigated the error between the real and estimated minimum with synthetic data. The data were generated from a cumulative distribution $F$ according to either of the fixed or i.i.d. data setting. We used the beta distribution to construct $F$. More precisely, let $[x_{\min},x_{\min} + \Delta]$ be the support of the data, and let $X$ be a random variable that follows the beta distribution with parameter $\alpha$ and $\beta$. Then, $F$ is the cumulative distribution of $x_{\min} + \Delta X$. $\Delta$, $a$, and $b$ are varied as combination of $\Delta \in \cbrace{0.3, 0.6, 0.9}$, $\alpha \in \cbrace{0.5, 0.9, 1, 2, 4}$, and $\beta \in \cbrace{1, 2}$. For stabilizing an error caused by discritization, we report the worst case mean absolute error among $x_{\min} \in \cbrace{0\times(2-\Delta)-1, 0.2\times(2-\Delta)-1, ..., 1\times(2-\Delta)-1}$. The mean absolute errors were calculated from average of 1000 runs. We also report the 0.05 and 0.95 quantiles of the errors.

We evaluated two different choices of $L$ and $h_N$ corresponding to \cref{cor:known-alpha,cor:unknown-alpha}:
\begin{description}[format=\boldparen]
 \item[Lower $\alpha$] $L = \ceil*{\log_2(N)/2}$ and $h_N = \ln(N)/2$,
 \item[Unknown $\alpha$] $L = \ceil*{\log_2^2(N)/2\log_2(1000)}$ and $h_N = \ln^2(N)/2\ln(1000)$.
\end{description}
The lower $\alpha$ case is a suitable parameter choice when the aggregator knows $\alpha \ge 1$, whereas the unknown $\alpha$ case is a suitable parameter choice when the aggregator has no information about $\alpha$. 

Here, we only show partial results in the fixed data setting such that $\alpha = 1$, $\beta = 1$, and $\Delta = 0.3$. Note that the beta distribution with $\alpha = \beta = 1$ is in fact the uniform distribution. More experiment results on different configurations can be found in the supplementary material.

{\bfseries Error v.s. $N$}) We first demonstrate that \cref{cor:known-alpha,cor:unknown-alpha} are tight with respect to both $N$. To this end, we evaluated the error of our mechanism corresponding to $N \in \cbrace{2^{10}, 2^{11}, ..., 2^{20}}$.

The left and middle figures in \cref{fig:err-vs-n-eps} show the errors of our proposed mechanism with varied $N$ and $\epsilon=1,4$. We choose $L$ and $h_N$ according to {\bfseries Lower $\alpha$} in the left and {\bfseries Unknown $\alpha$} in the middle, respectively. The blue lines denote the theoretical guidelines from \cref{cor:known-alpha,cor:unknown-alpha}. We can see from \cref{fig:err-vs-n-eps} that the slopes of the errors are almost the same as the slope of the theoretical guideline regardless of choice of $\epsilon$ in both {\bfseries Lower $\alpha$} and {\bfseries Unknown $\alpha$}. This indicates that the decreasing rates with respect to $N$ shown in \cref{cor:known-alpha,cor:unknown-alpha} are tight.

{\bfseries Error v.s. $\epsilon$}) Next, we show tightness of \cref{cor:known-alpha,cor:unknown-alpha} regarding $\epsilon$. To this end, we evaluated the error of our mechanism corresponding to $\epsilon \in \cbrace{2^{-3}, 2^{-2}, ..., 2^6}$.

The right figure in \cref{fig:err-vs-n-eps} shows the errors of our proposed mechanism with varying $\epsilon$ and $N=2^{15},2^{20}$. The yellow line represents the theoretical guideline from \cref{cor:known-alpha,cor:unknown-alpha}. If $\epsilon$ is not large, slopes of the error are almost the same as the slope of the theoretical guideline, where the error is saturated up to $2$ for small $\epsilon$ since the data are supported on $[-1,1]$. We therefore can conclude that the rates in \cref{cor:known-alpha,cor:unknown-alpha} with respect to $\epsilon$ are tight in the range of small $\epsilon$. Looseness in large $\epsilon$ comes from \cref{thm:adjusted}. When deriving \cref{thm:adjusted}, we use a bound $\paren{\nicefrac{e^{\epsilon/L}(1+e^{\epsilon/L})}{(e^{\epsilon/L}-1)^2}}^{1/2\alpha} \le \paren{\nicefrac{2L^2}{\epsilon^2}}^{1/2\alpha}$, which is valid only if $\epsilon$ is sufficiently small. The experimental results reflect this behavior.

In both experiments of error v.s. $N$ and $\epsilon$, the rate looks faster than the theoretical guideline when both $N$ and $\epsilon$ is small. This is acceptable because the big-O notation in \cref{thm:adjusted} indicates that the rate is satisfied only if $\nicefrac{L^2h_N}{\epsilon^2N}$ is sufficiently small. 

{\bfseries Comparison with Naive Mechanism}) We also carried out empirical comparison between our proposed method and a baseline solution. Since there is no existing locally private method for finding the minimum, we consider the straightforward Laplace method as a baseline. In particular, each user with $x_i$ reports $\hat{x}_i=x_i+\delta_i$ with $\delta_i \sim \mathcal{L}(0, 2/\epsilon)$, where $\mathcal{L}(\mu, b)$ is the Laplace distribution with mean $\mu$ and scale parameter $b$. The server simply considers the $\min_i \hat{x}_i$ as the estimated minimum. In this experiment, we use $N \in \cbrace{2^{10}, 2^{11}, ..., 2^{20}}$ and $\epsilon \in \cbrace{1, 4}$.

The comparison between our method and the baseline method is shown in \cref{fig:cmp}. We can see from \cref{fig:cmp} that the baseline mechanism suffers from an error larger than $1$ for all $N$. Since the data are supported on $[-1,1]$, the baseline mechanism fails in reasonable estimation. On the other hand, our proposed mechanism achieves significantly smaller error than the baseline method and successes to decrease its error as $N$ increases. 

\subsection{MovieLens Data}

\begin{figure}[tbhp]
  \centering
  \begin{subfigure}{.4\textwidth}
    \includegraphics[width=\textwidth]{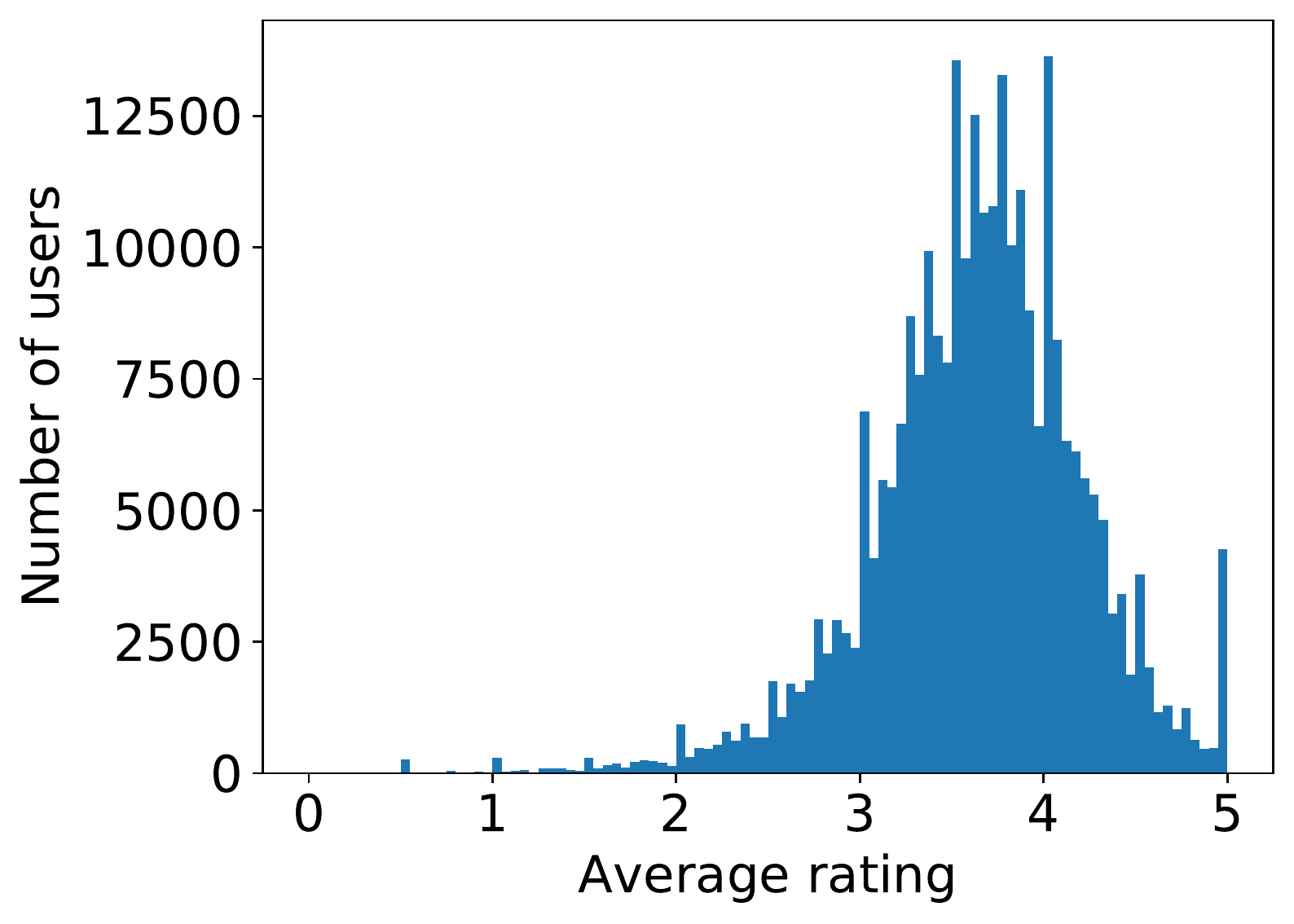}
    \caption{{\bfseries Task1}}
  \end{subfigure}
  \begin{subfigure}{.4\textwidth}
    \includegraphics[width=\textwidth]{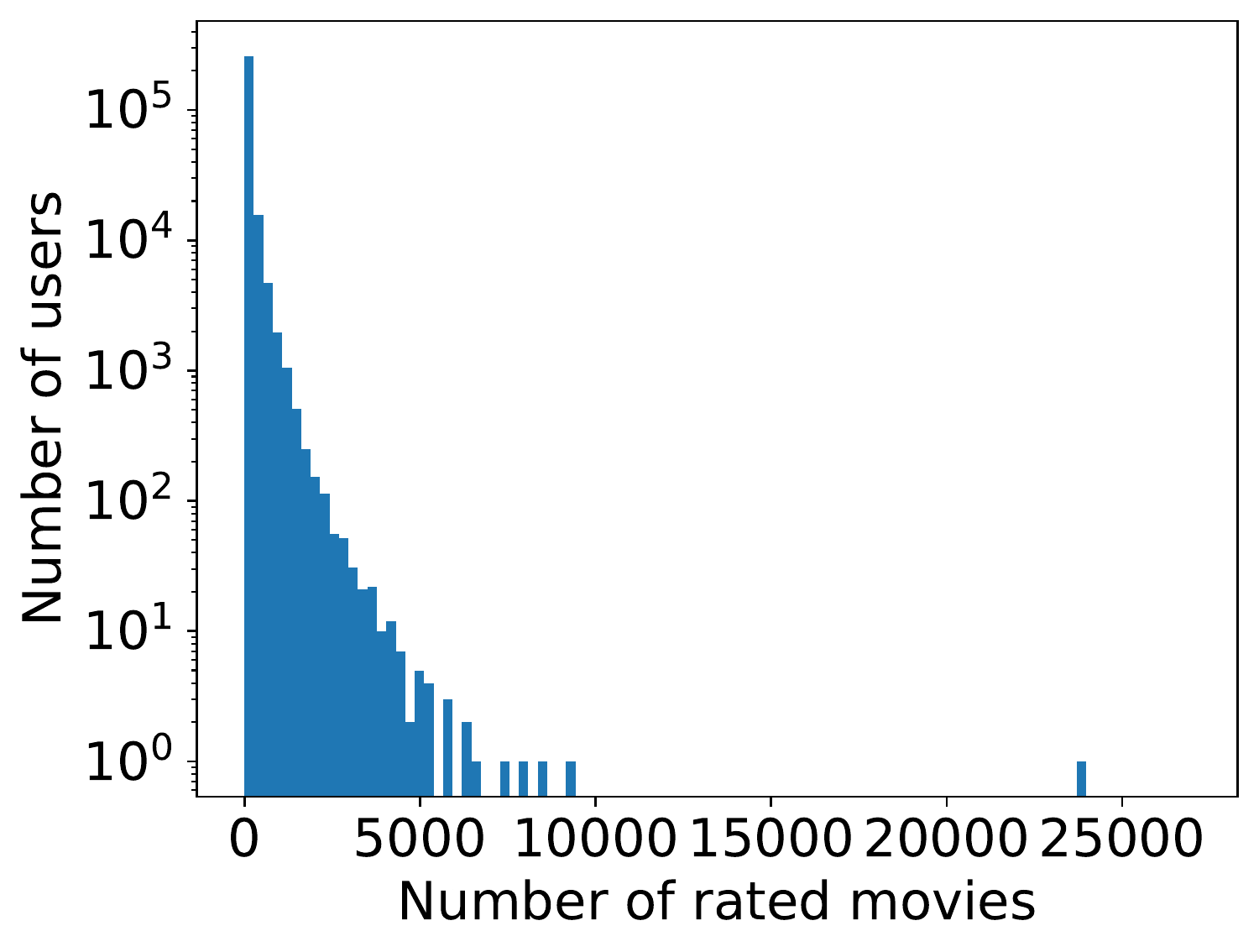}
    \caption{{\bfseries Task2}}
  \end{subfigure}
    \caption{Histogram of the MovieLens dataset for each tasks. Note that the horizontal axis of the right figure is log-scale.}
    \label{fig:hist}
\end{figure}
\begin{figure}[tb]
  \centering
  \begin{subfigure}{.35\columnwidth}
    \includegraphics[width=\textwidth]{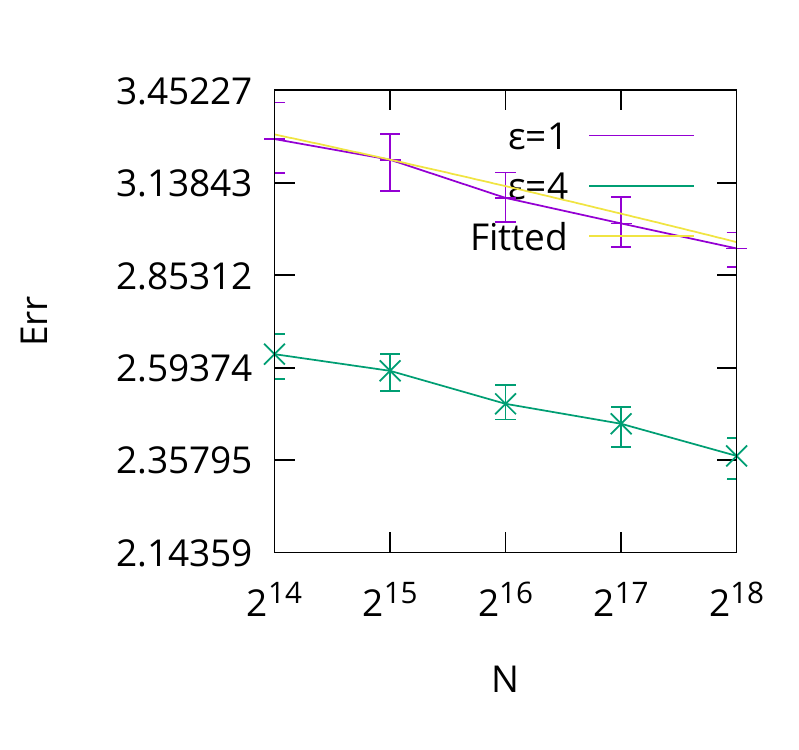}
    \caption{{\bfseries Task1} $\min$: $\alpha = 7.69$}
  \end{subfigure}
  \begin{subfigure}{.35\columnwidth}
    \includegraphics[width=\textwidth]{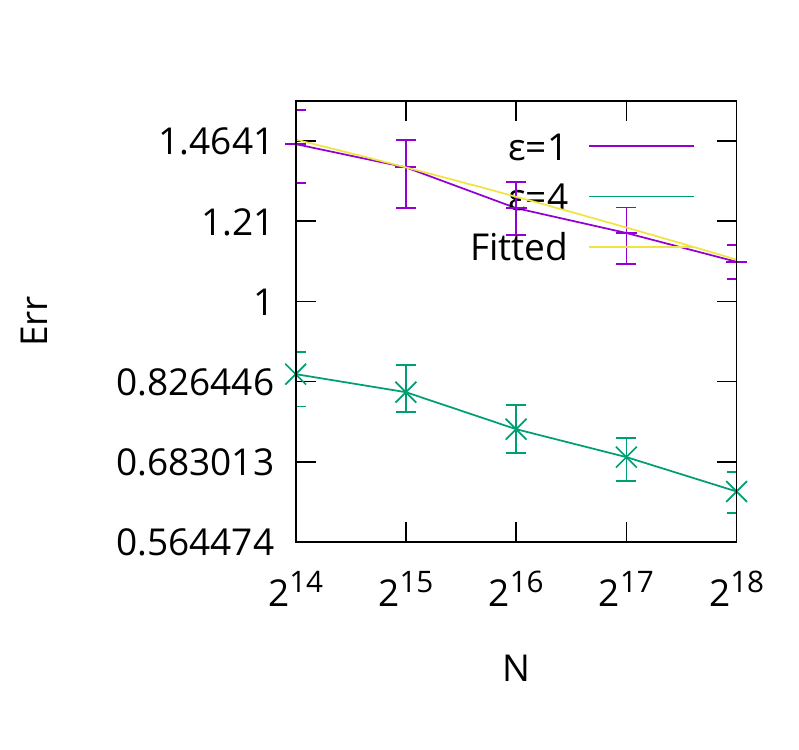}
    \caption{{\bfseries Task1} $\max$: $\alpha = 2.50$}
  \end{subfigure}\\
  \begin{subfigure}{.35\columnwidth}
    \includegraphics[width=\textwidth]{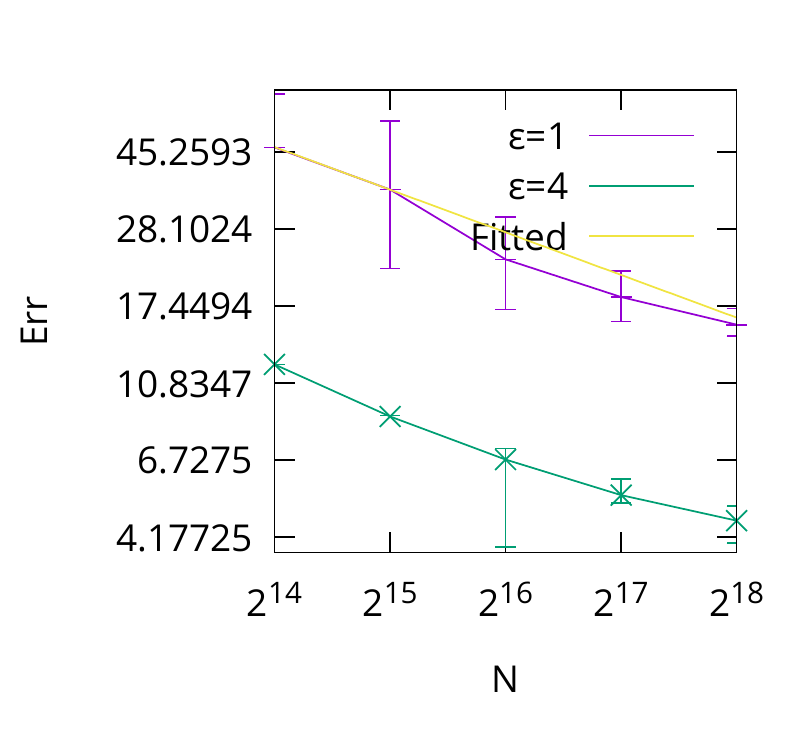}
    \caption{{\bfseries Task2} $\min$: $\alpha = 1.31$}
  \end{subfigure}
  \begin{subfigure}{.35\columnwidth}
    \includegraphics[width=\textwidth]{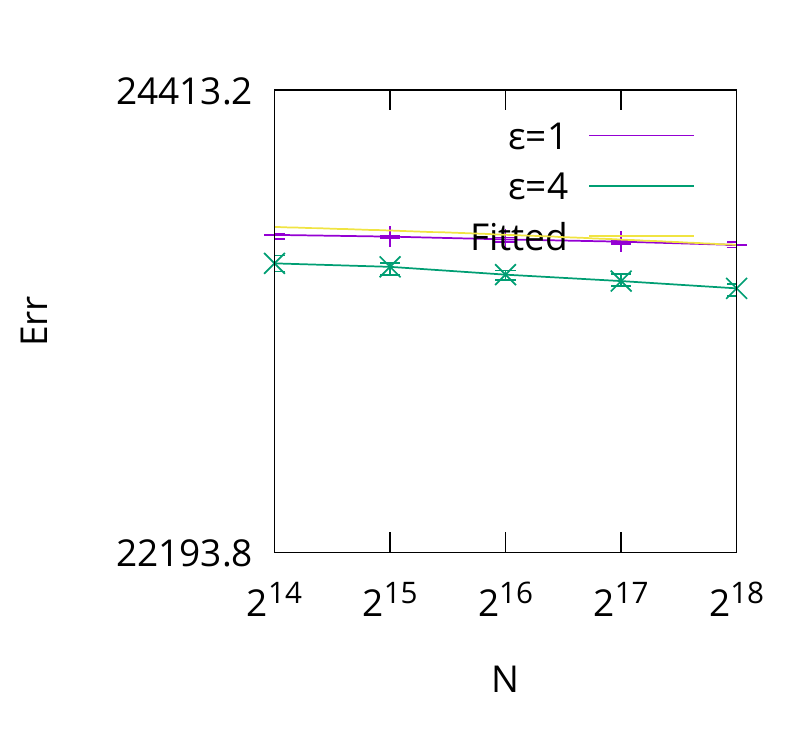}
    \caption{{\bfseries Task2} $\max$: $\alpha = 113$}
  \end{subfigure}
    \caption{$\mathrm{Err}$ v.s. $N$ on the MovieLens dataset. The yellow line denotes a function $N \to C\log^BN/N^A$ where $A$ and $B$ are obtained by the least square method. We show the value of $\alpha = 1/2A$ in the subcaptions. }
    \label{fig:real-results}
\end{figure}

We conducted experiments on the MovieLens dataset\footnote{Available at \url{https://grouplens.org/datasets/movielens/latest/}}. We used the full dataset consisting of 27,753,444 ratings for 53,889 movies obtained by 283,228 users. We carried out the following tasks. ({\bfseries Taks1}) the server estimates the minimum and maximum of the users' average rating. The domain of the rating is $[0,5]$. ({\bfseries Task2}) the server estimates the minimum and maximum numbers of the rated movies per user. We can naturally assume that no user exists that evaluate all the movies. We here assumed that the number of the movies rated by a single user was within [0, 53,889/2]. We evaluated the error of our mechanism with varying $N \in \cbrace{2^{14}, ..., 2^{18}}$ by subsampling the dataset, where we use $\epsilon \in \cbrace{1, 4}$. Since $\alpha$, the fatness, of the distributions is unknown, we used the {\bfseries Unknown $\alpha$} parameter setting shown in \cref{sec:syn}. The reported value is an average of 1000 runs. We also report the 0.05 and 0.95 quantiles of the errors.

{\bfseries Results})
The histograms of the dataset for each task are dipicted in \cref{fig:hist}. As shown in \cref{fig:hist}, the left-side tail of the average review distribution is longer than the right-side tail. Regarding the distribution of the number of reviews per user, the right-side tail is extremely long compared to the left-side tail. We, therefore, can expect that in {\bfseries Taks 1}, $\alpha$ of the left-side tail is larger than that of the right-side tail, and in {\bfseries Task 2}, $\alpha$ of the right-side tail is extremely larger than that of the right-side tail. 

\cref{fig:real-results} shows the experimental results. We can see from \cref{fig:real-results} that the decreasing rates of the estimation error are changed adaptively to $\alpha$, and the obtained $\alpha$ shown in the subcaptions corresponds to the fatness of the tail. 

\subsection{Purchase History Dataset}

\begin{figure}[tb]
  \centering
  \begin{subfigure}{.35\columnwidth}
    \includegraphics[width=\textwidth]{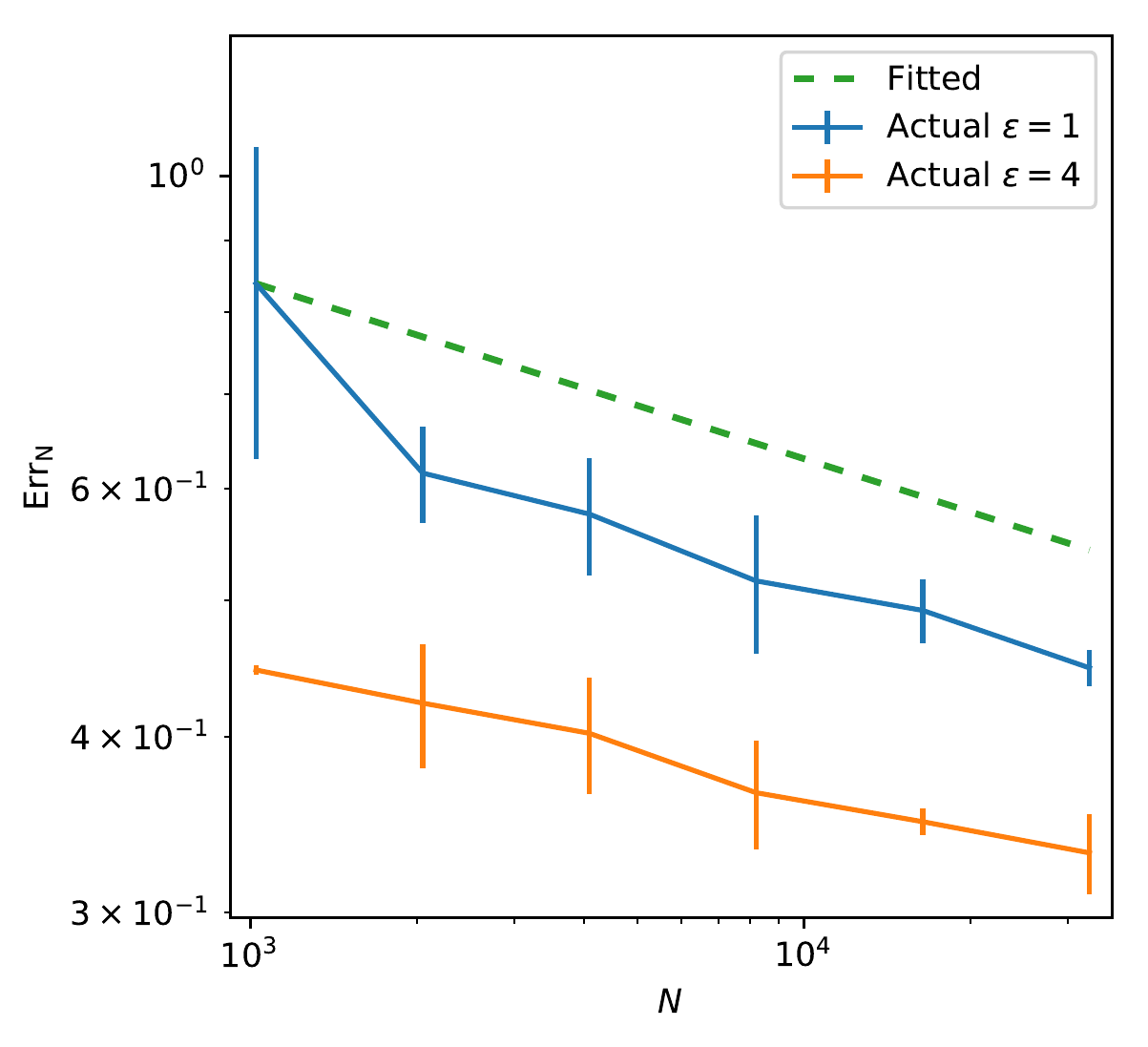}
    \caption{{\bfseries Task1}: $\alpha = 3.98$}
  \end{subfigure}
  \begin{subfigure}{.35\columnwidth}
    \includegraphics[width=\textwidth]{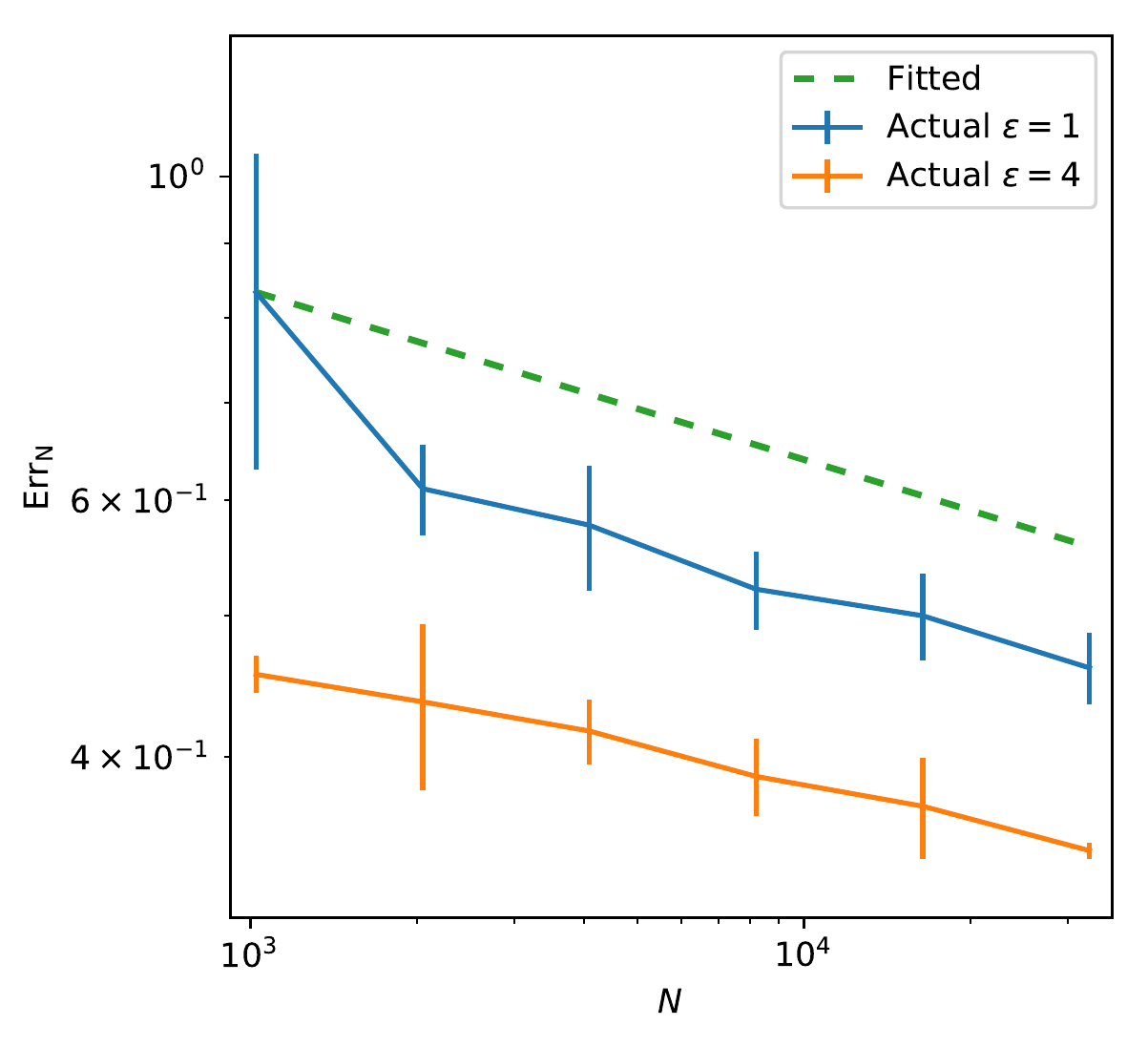}
    \caption{{\bfseries Task1}: $\alpha=4.31$}
  \end{subfigure}\\
  \begin{subfigure}{.35\columnwidth}
    \includegraphics[width=\textwidth]{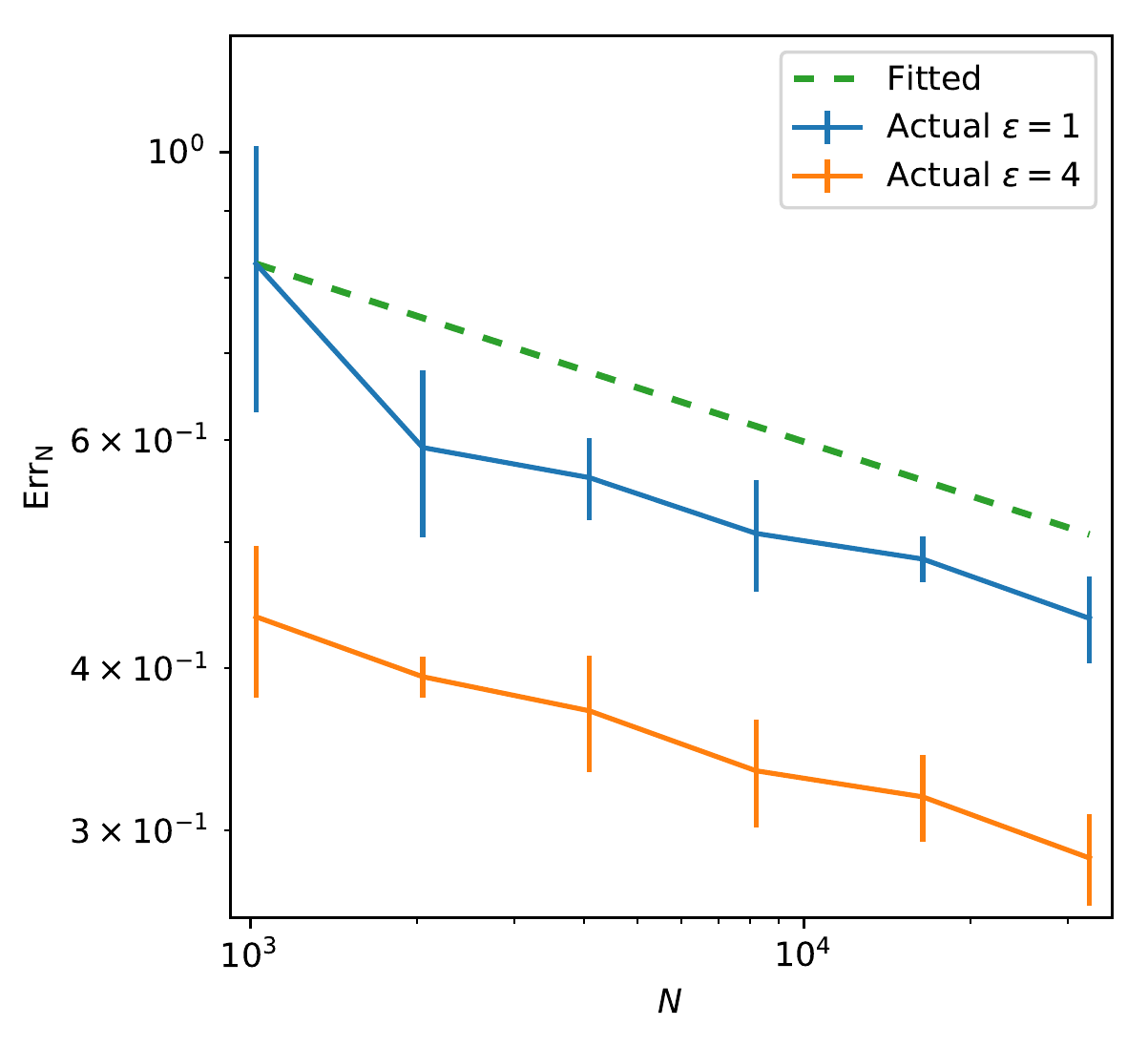}
    \caption{{\bfseries Task2}: $\alpha=3.60$}
  \end{subfigure}
  \begin{subfigure}{.35\columnwidth}
    \includegraphics[width=\textwidth]{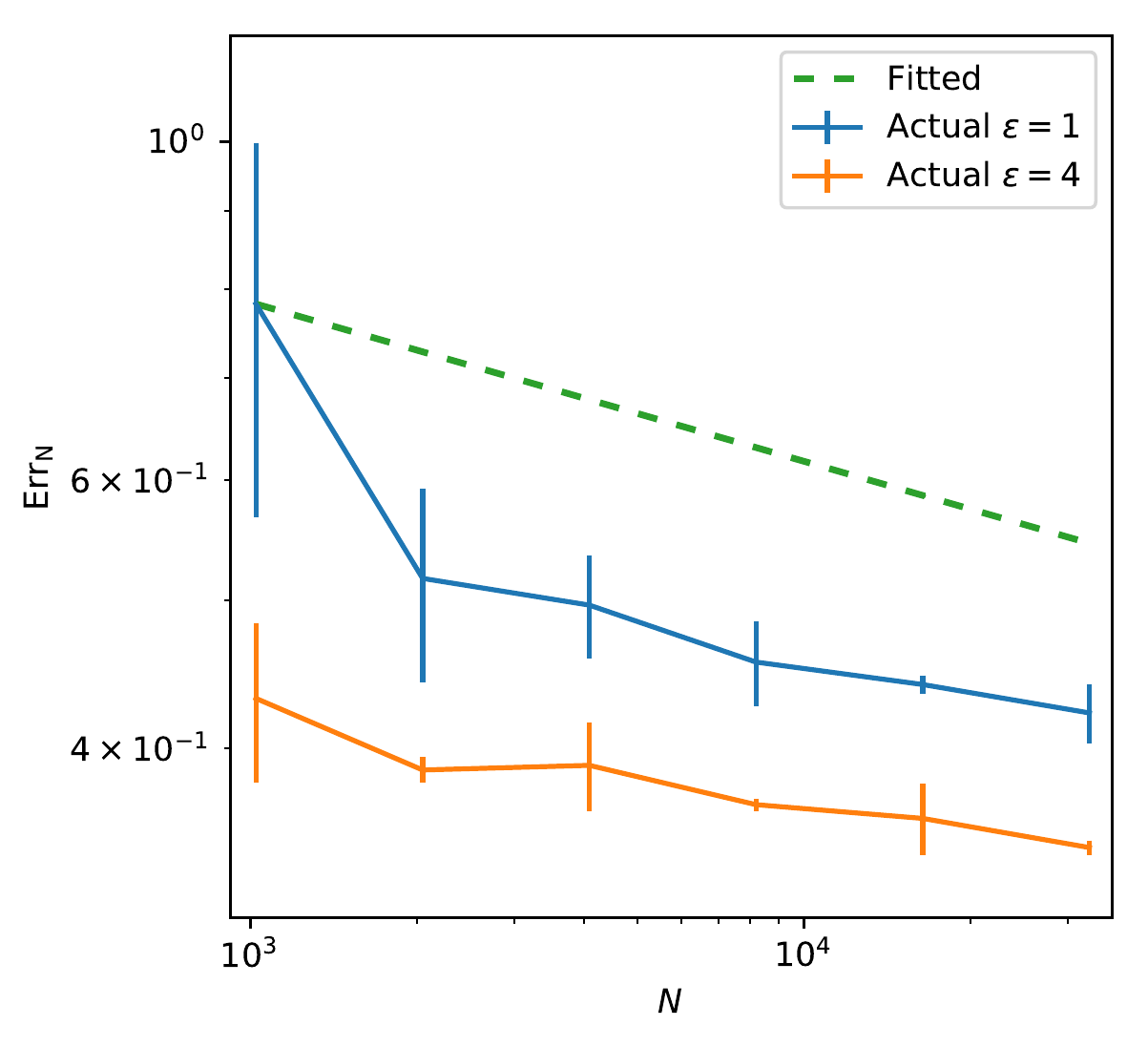}
    \caption{{\bfseries Task2}: $\alpha=4.79$}
  \end{subfigure}
    \caption{$\mathrm{Err}$ v.s. $N$ on the purchase history dataset. The dotted line represents a function $N \to C\log^BN/N^A$ where $A$ and $B$ are obtained by the least square method. We show the concrete value of $\alpha = 1/2A$ in the subcaptions. }
   \label{fig:yahoo-results}
\end{figure}

We also conducted experiments on a purchase history dataset collected in the shopping service provided by Yahoo Japan Corporation. This dataset consists of user attribute information, such as gender and birthday. Also, the dataset contains histories of purchase orders of users in Dec. 2015. Each order consists of a multiset of items purchased. We carried out the following tasks with this dataset:
\begin{itemize}
 \item {\bfseries Task1}: The server estimates the minimum age of users whose total amount of purchase on this month was in some range. The ranges are varied as [\yen 0, \yen 10,000], [\yen 10,000,\yen 20,000],  [\yen 20,000,\yen 30,000], [\yen 30,000,\yen 40,000], [\yen 40,000,\yen 50,000], and [\yen 50,000,\yen 60,000]. 
 \item {\bfseries Task2}: The server estimates the minimum age of the users who purchased items in a specific product category.
\end{itemize}
Here, the age is rescaled from $[0,150]$ to $[-1,1]$. The items are categorized into 23 types of products~(e.g. fashion, food, sports), whereas only 19 categories were used so that the number of users who purchased an item in a category is larger than $2^{15}$. We evaluated the error of our mechanism with varying $N \in \cbrace{2^{11}, ..., 2^{15}}$ by subsampling the dataset, where we use $\epsilon \in \cbrace{1, 4}$. Since alpha, the fatness, of the distributions are unknown, we used the {\bfseries Unknown $\alpha$} parameter setting shown in \cref{sec:syn}. The reported value is an average of 1000 runs. We also report the 0.05 and 0.95 quantiles of the errors.

{\bfseries Results})
\cref{fig:yahoo-results} shows the experimental results with the real datasets. The figure only consists of the results for {\bfseries Task 1} with the ranges [\yen 0, \yen 10,000]~(left) and [\yen 40,000, \yen 50,000]~(right), {\bfseries Task 2} with the categories music-software~(left) and baby-kids-maternity~(right). The results with the other settings can be found in the appendix. 

We can see from \cref{fig:yahoo-results} that for these task, the estimation error of our proposed mechanism decreases as $N$ increases. Thus, we can expect that our mechanism can be consistently estimate the minimum in the real data. Furthermore, the decreasing rates of the estimation error are changed adaptively to the ranges~(in {\bfseries Task 1}) and categories~(in {\bfseries Task 2}).

\section{Related Work}\label{sec:related}
LDP gains the first real-world application in Google Chrome's extension, RAPPOR~\citep{Erlingsson2013ccs} and thereafter also finds applications on the other problems such as distribution estimation~\citep{Erlingsson2013ccs,Fanti2016popet,Ren2018tifs} and heavy hitter estimation~\citep{Bassily2015stoc} for categorical-valued data. Different from existing works, our proposed method addresses finding the minimum over numeric-valued data.  Simply bucketizing numeric-valued data as categorical data introduces the estimation error. Thus, to handle numeric-valued data, more elaborate protocol design and analysis are needed. There are also local differential privacy methods for numeric-valued problems. For example, \citet{Ding2017nips}, \citet{Duchi2013LocalRates}, and \citet{Nguyen2016arxiv} estimate the mean of numeric-valued data under LDP. \citet{Ding2018aaai} studied hypothesis testing to compare population means while preserving privacy. \citet{Kairouz2016jmlr} studied the optimal trade-off between privacy and utility. However, these techniques deal with fundamentally different problems from ours and cannot be extended to the minimum finding problem easily. 

Essentially, our proposed method adopts binary search-based strategy, together with randomized response, to find the minimum. \citet{Cyphers2017dsaa} developed \textsf{AnonML} to estimate the median over real-valued data under LDP. This method shares the same spirit with ours, i.e., binary search-based strategy with randomized response. However, the estimation error of their mechanism was not analyzed, for which we cannot set the number of rounds for binary search reasonably. 

\citet{Gaboardi2019} dealt with a problem of estimating a mean with a confidence bound from a set of i.i.d. Gaussian samples. When calculating the confidence bound, they utilize a locally private quantile estimation mechanism that is almost the same as the one we propose. An utility analysis of the quantile estimation mechanism was also provided by them; however, their analysis does not necessarily guarantee a small estimation error. This is because that their analysis employs two utility criteria and provides the sample complexity to achieve that \emph{either} of them is small. More precisely, their utility criteria for $p$-quantile estimation are $\abs{\tilde{x} - F^*(p)}$ and $\abs{F(\tilde{x}) - p}$, where $\tilde{x}$ denotes the estimated value. Their utility analysis provides the sample complexity to achieve either of $\abs{\tilde{x} - F^*(p)} \le \tau$ or $\abs{F(\tilde{x}) - p} \le \lambda$. Noting that the former criterion is the absolute error of the $p$-quantile estimation, we can see that their analysis does not guarantee small error of $p$-quantile estimation if $\abs{F(\tilde{x}) - p}$ is small. 

Our minimum finding mechanism~(which can be easily adapted to maximum finding) can be employed as a preprocessing for various types of locally differentially private data analysis. For example, we can use our method for locally differentially private itemset mining~\citep{Qin2016ccs, Wang2018sp} over set-valued data. The crucial assumption employed for these methods is that the server knows the maximum number of data items owned by users. The maximum number can be estimated by our mechanism in a local differential privacy manner.

\section{Conclusion}
In this study, we propose a method for finding the minimum of individuals' data values under local differential privacy. We reveal that the absolute error of the proposed mechanism is $O((\nicefrac{\ln^6N}{\epsilon^2N})^{1/2\alpha})$ under the $\alpha$-fatness assumption and it is optimal up to log factors. The mechanism is adaptive to $\alpha$; that is, it can obtain that rate without knowing $\alpha$. 


\subsubsection*{Acknowledgments}
This work was partly supported by KAKENHI (Grants-in-Aid for
scientific research) Grant Numbers JP19H04164 and JP18H04099. We would like to express our gratitude to Yahoo Japan Corporation for providing the purchase history dataset. 

\bibliographystyle{plainnat}
\bibliography{mendeley_v2}

\begin{thebibliography}{18}
\providecommand{\natexlab}[1]{#1}
\providecommand{\url}[1]{\texttt{#1}}
\expandafter\ifx\csname urlstyle\endcsname\relax
  \providecommand{\doi}[1]{doi: #1}\else
  \providecommand{\doi}{doi: \begingroup \urlstyle{rm}\Url}\fi

\bibitem[Bassily and Smith(2015)]{Bassily2015stoc}
Raef Bassily and Adam~D. Smith.
\newblock {Local, Private, Efficient Protocols for Succinct Histograms}.
\newblock In \emph{Proceedings of the Forty-Seventh Annual {ACM} on Symposium
  on Theory of Computing (STOC)}, pages 127--135. {ACM}, 2015.
\newblock \doi{10.1145/2746539.2746632}.

\bibitem[Boucheron et~al.(2003)Boucheron, Lugosi, and
  Massart]{Boucheron2003ConcentrationMethod}
St\'ephane Boucheron, G\'abor Lugosi, and Pascal Massart.
\newblock {Concentration inequalities using the entropy method}.
\newblock \emph{The Annals of Probability}, 31\penalty0 (3):\penalty0
  1583--1614, 7 2003.
\newblock \doi{10.1214/aop/1055425791}.

\bibitem[Cyphers and Veeramachaneni(2017)]{Cyphers2017dsaa}
Bennett Cyphers and Kalyan Veeramachaneni.
\newblock {AnonML: Locally private machine learning over a network of peers}.
\newblock In \emph{IEEE International Conference on Data Science and Advanced
  Analytics (DSAA)}, pages 549--560. {IEEE}, 2017.
\newblock \doi{10.1109/DSAA.2017.80}.

\bibitem[Ding et~al.(2017)Ding, Kulkarni, and Yekhanin]{Ding2017nips}
Bolin Ding, Janardhan Kulkarni, and Sergey Yekhanin.
\newblock {Collecting telemetry data privately}.
\newblock In \emph{Advances in Neural Information Processing Systems 30: Annual
  Conference on Neural Information Processing Systems (NIPS)}, pages
  3574--3583, 2017.

\bibitem[Ding et~al.(2018)Ding, Nori, Li, and Joshua]{Ding2018aaai}
Bolin Ding, Harsha Nori, Paul Li, and Allen Joshua.
\newblock {Comparing Population Means under Local Differential Privacy: with
  Significance and Power}.
\newblock In \emph{Proceedings of the Thirty-Second {AAAI} Conference on
  Artificial Intelligence}, pages 26--33. {AAAI} Press, 2018.

\bibitem[Duchi et~al.(2016)Duchi, Wainwright, and Jordan]{Duchi2016}
John Duchi, Martin Wainwright, and Michael Jordan.
\newblock {Minimax Optimal Procedures for Locally Private Estimation}.
\newblock \emph{ArXiv e-prints}, 2016.
\newblock URL \url{http://arxiv.org/abs/1604.02390}.

\bibitem[Duchi et~al.(2013)Duchi, Jordan, and Wainwright]{Duchi2013LocalRates}
John~C. Duchi, Michael~I. Jordan, and Martin~J. Wainwright.
\newblock {Local privacy and statistical minimax rates}.
\newblock In \emph{54th Annual {IEEE} Symposium on Foundations of Computer
  Science (FOCS)}, pages 429--438, 2013.
\newblock \doi{10.1109/FOCS.2013.53}.

\bibitem[Dwork et~al.(2006)Dwork, McSherry, Nissim, and Smith]{Dwork2006}
Cynthia Dwork, Frank McSherry, Kobbi Nissim, and Adam Smith.
\newblock Calibrating noise to sensitivity in private data analysis.
\newblock In Shai Halevi and Tal Rabin, editors, \emph{Theory of Cryptography},
  pages 265--284, Berlin, Heidelberg, 2006. Springer Berlin Heidelberg.
\newblock ISBN 978-3-540-32732-5.

\bibitem[Erlingsson et~al.(2013)Erlingsson, Pihur, and
  Korolova]{Erlingsson2013ccs}
\'{U}lfar Erlingsson, Vasyl Pihur, and Aleksandra Korolova.
\newblock {{RAPPOR:} Randomized Aggregatable Privacy-Preserving Ordinal
  Response}.
\newblock In \emph{Proceedings of the 2014 {ACM} {SIGSAC} Conference on
  Computer and Communications Security (CCS)}, pages 1054--1067. {ACM}, 2013.
\newblock \doi{10.1145/2660267.2660348}.

\bibitem[Evfimievski et~al.(2003)Evfimievski, Gehrke, and
  Srikant]{Evfimievski2003}
Alexandre Evfimievski, Johannes Gehrke, and Ramakrishnan Srikant.
\newblock Limiting privacy breaches in privacy preserving data mining.
\newblock In \emph{Proceedings of the twenty-second ACM SIGMOD-SIGACT-SIGART
  symposium on Principles of database systems}, pages 211--222. ACM, 2003.

\bibitem[Fanti et~al.(2016)Fanti, Pihur, and Erlingsson]{Fanti2016popet}
Giulia~C. Fanti, Vasyl Pihur, and \'{U}lfar Erlingsson.
\newblock {Building a RAPPOR with the Unknown: Privacy-Preserving Learning of
  Associations and Data Dictionaries}.
\newblock \emph{Proceedings on Privacy Enhancing Technologies (PoPETs)},
  2016\penalty0 (3):\penalty0 41--61, 2016.

\bibitem[Gaboardi et~al.(2019)Gaboardi, Rogers, and Sheffet]{Gaboardi2019}
Marco Gaboardi, Ryan Rogers, and Or~Sheffet.
\newblock Locally private mean estimation: $z$-test and tight confidence
  intervals.
\newblock In Kamalika Chaudhuri and Masashi Sugiyama, editors,
  \emph{Proceedings of Machine Learning Research}, volume~89 of
  \emph{Proceedings of Machine Learning Research}, pages 2545--2554. PMLR,
  2019.

\bibitem[Kairouz et~al.(2016)Kairouz, Oh, and Viswanath]{Kairouz2016jmlr}
Peter Kairouz, Sewoong Oh, and Pramod Viswanath.
\newblock {Extremal Mechanisms for Local Differential Privacy}.
\newblock \emph{Journal of Machine Learning Research (JMLR)}, 17:\penalty0
  17:1--17:51, 2016.

\bibitem[Nguy\^{e}n et~al.(2016)Nguy\^{e}n, Xiao, Yang, Hui, Shin, and
  Shin]{Nguyen2016arxiv}
Th\^{o}ng~T. Nguy\^{e}n, Xiaokui Xiao, Ying Yang, Siu~Cheung Hui, Hyejin Shin,
  and Junbum Shin.
\newblock {Collecting and analyzing data from smart device users with local
  differential privacy}.
\newblock \emph{ArXiv e-prints}, 2016.
\newblock URL \url{https://arxiv.org/abs/1606.05053}.

\bibitem[Qin et~al.(2016)Qin, Yang, Yu, Khalil, Xiao, and Ren]{Qin2016ccs}
Zhan Qin, Yin Yang, Ting Yu, Issa Khalil, Kiaokui Xiao, and Kui Ren.
\newblock {Heavy hitter estimation over set-valued data with local differential
  privacy}.
\newblock \emph{ACM Conference on Computer and Communications Security (CCS)},
  2016.

\bibitem[Ren et~al.(2018)Ren, Yu, Yu, Yang, Yang, McCann, and Yu]{Ren2018tifs}
Xuebin Ren, Chia-Mu Yu, Weiren Yu, Shusen Yang, Xinyu Yang, Julie~A. McCann,
  and Philip~S. Yu.
\newblock {LoPub: High-Dimensional Crowdsourced Data Publication With Local
  Differential Privacy}.
\newblock \emph{IEEE Transactions on Information Forensics and Security},
  13\penalty0 (9):\penalty0 2151--2166, 2018.
\newblock \doi{10.1109/TIFS.2018.2812146}.

\bibitem[Wang et~al.(2018)Wang, Li, and Jha]{Wang2018sp}
Tianhao Wang, Ninghui Li, and Somesh Jha.
\newblock {Locally Differentially Private Frequent Itemset Mining}.
\newblock In \emph{IEEE Symposium on Security and Privacy (S\&P)}, pages
  127--143. {IEEE}, 2018.
\newblock \doi{10.1109/SP.2018.00035}.

\bibitem[Warner(1965)]{Warner1965RandomizedBias}
Stanley~L. Warner.
\newblock {Randomized Response: A Survey Technique for Eliminating Evasive
  Answer Bias}.
\newblock \emph{Journal of the American Statistical Association}, 1965.
\newblock \doi{10.1080/01621459.1965.10480775}.

\end{thebibliography}

\appendix

\section{Examples of $\alpha$-Fat Distributions}

\begin{figure}[tbhp]
    \centering
    \includegraphics[width=.4\columnwidth]{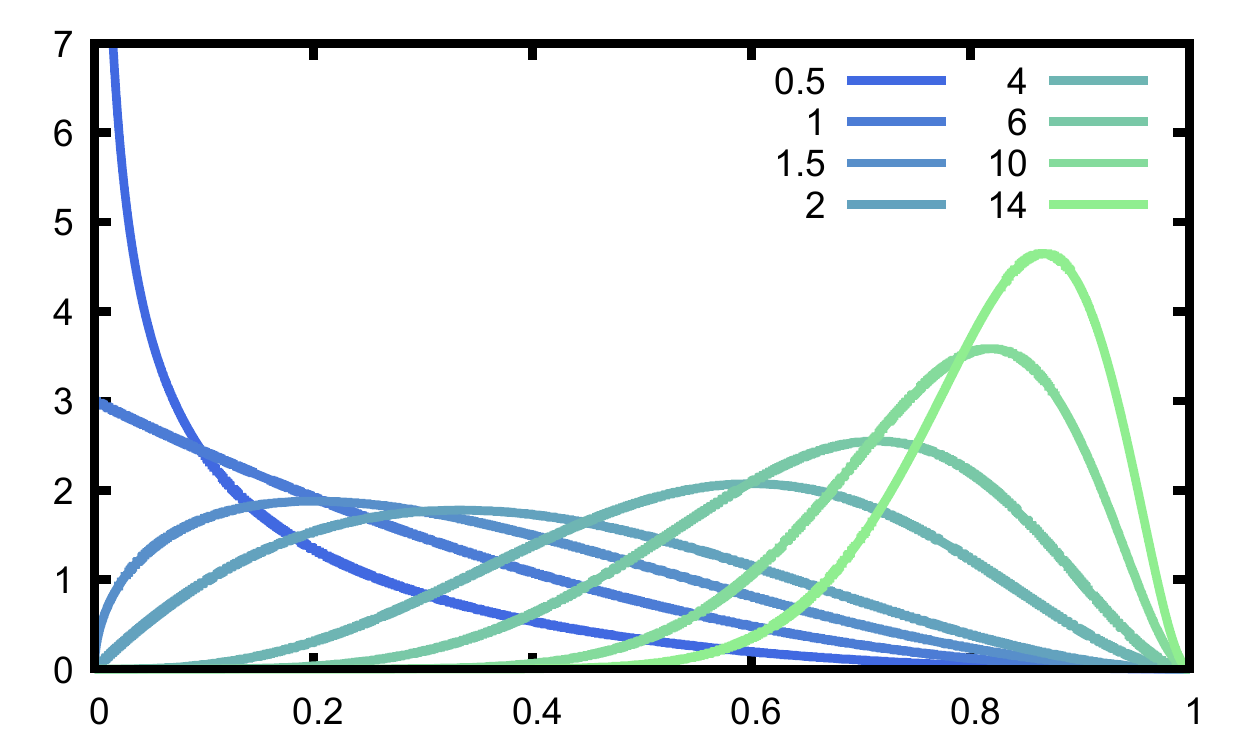}
    \caption{The density function of the beta distribution with $\alpha \in \cbrace{0.5, 1, 1.5, 2, 4, 6, 10, 14}$ and $\beta=3$.}
    \label{fig:beta}
\end{figure}


For giving a better understanding of $\alpha$-fatness, we introduce some concrete values of $\alpha$, $C$, and $\bar{x}$ for some $F$. 
\begin{example}[Beta distribution]
  Let $X$ be a random variable following the beta distribution with parameters $\alpha$ and $\beta$. Suppose $F$ is the cumulative distribution of $x_{\min} + (x_{\max} - x_{\min})X$. Then, \cref{def:fatness} is satisfied with the same $\alpha$, and with $C = \max\cbrace{1, (\alpha \Beta(\alpha,\beta))^{-1}}/(x_{\max} - x_{\min})^\alpha$ and $\bar{x} = x_{\max}$, where $\Beta(\alpha,\beta)$ denotes the beta function.
\end{example}
\begin{example}[Truncated (normal) distributions]
  Suppose $F$ is the cumulative distribution of the truncated normal distribution supported on $[x_{\min},x_{\max}]$ with parameters $\mu \in [-1,1]$ and $\sigma^2 > 0$. Then, \cref{def:fatness} is satisfied with $\alpha=1$ and $C=\min\cbrace{\phi((x_{\min}-\mu)/\sigma),\phi((x_{\max}-\mu)/\sigma)}/\sigma(\Phi((x_{\max} - \mu)/\sigma)-\Phi((x_{\min} - \mu)/\sigma)$, where $\phi(x)$ and $\Phi(x)$ denote a density function and cumulative function of the standard normal distribution, respectively. More generally, any truncated distribution satisfies \cref{def:fatness} with $\alpha=1$.
\end{example}
\cref{fig:beta} shows the probability density function of the beta distribution with different parameter settings. In \cref{fig:beta}, we set $\beta = 3$, and $\alpha$ are varied as shown in the legend. We can see from \cref{fig:beta} that density around the minimum becomes larger as $\alpha$ decreases.

\section{Analyses of \cref{alg:find-min}}

\subsection{Utility Analysis}
Here, we will prove the following two theorems corresponding to the fixed and i.i.d. data settings.
\begin{theorem}\label{thm:non-adjust}
 Suppose $F$ is $\alpha$-fat, and $\gamma$ satisfies
 \begin{align}
     2\gamma < C\paren*{\bar{x} - x_{\min}}^\alpha.
 \end{align}
 In the fixed data setting, for any $N$, we have
 \begin{align}
     \mathrm{Err} \le 2\paren*{\frac{2\gamma}{C}}^{1/\alpha} + \exp\paren*{-\frac{(e^{\epsilon/L}-1)^2\gamma^2N}{4(e^{\epsilon/L}+1)e^{\epsilon/L}}} + 2^{-L}.
 \end{align}
\end{theorem}
\begin{theorem}\label{thm:non-adjust-iid}
 Suppose $F$ is $\alpha$-fat, and $\gamma$ satisfies
 \begin{align}
     2\gamma < C\paren*{\bar{x} - x_{\min}}^\alpha.
 \end{align}
 In the i.i.d. data setting, for any $N$, we have
 \begin{align}
     \mathrm{Err} \le 2\paren*{\frac{1}{C}}^{1/\alpha}\frac{(\ceil{2\gamma N})^{\overline{1/\alpha})}}{(N+1)^{\overline{1/\alpha}}} + \exp\paren*{-\frac{(e^{\epsilon/L}-1)^2\gamma^2N}{4(e^{\epsilon/L}+1)e^{\epsilon/L}}} + 2^{-L},
 \end{align}
 where $(x)^{\overline{\alpha}}$ denotes the rising factorial. That is, letting $\Gamma$ be the gamma function, $(x)^{\overline{\alpha}} = \Gamma(x+\alpha)/\Gamma(x)$.
\end{theorem}
We obtain \cref{thm:adjusted} by substituting the specified $\gamma$. Note that the first term in \cref{thm:non-adjust-iid} matches the first term in \cref{thm:adjusted} because if $\gamma = \omega(1/N)$, we have
\begin{align}
    \lim_{N \to \infty}\frac{(\ceil{2\gamma N})^{\overline{1/\alpha})}}{(N+1)^{\overline{1/\alpha}}}\paren*{\frac{N+1}{2\gamma N}}^{1/\alpha} = 1.
\end{align}

Here, we give the proof sketch of \cref{thm:non-adjust,thm:non-adjust-iid}. \cref{alg:find-min} can be seen as an algorithm that estimates $\gamma$-quantile of the users' data because the algorithm finds $x \in [-1,1]$ such that $\tilde{F}(x) = \gamma$. Hence, the mean absolute error of \cref{alg:find-min} can be decomposed as
\begin{align}
    \mathrm{Err} \le \Mean\bracket*{\abs*{\tilde{F}^*(\gamma) - x_{\min}}} + \Mean\bracket*{\abs*{\tilde{x} - \tilde{F}^*(\gamma)}}. \label{eq:decomp}
\end{align}
The first term in \cref{eq:decomp} denotes the error between the minimum and $\gamma$-quantile, and the second term denotes the estimation error of the $\gamma$-quantile.

To analyze the second term in \cref{eq:decomp}, we define events of {\em mistake}. For each round $t$, define an event
\begin{align}
 \mathcal{M}_t = \cbrace*{\tau_t < \tilde{F}^*(\gamma) \implies \Phi(z) \ge \gamma}\cap\cbrace*{\tau_t > \tilde{F}^*(\gamma) \implies \Phi(z) < \gamma}.
\end{align}
Then, $\mathcal{M}_t$ represents an event that, at round $t$, the algorithm chooses an interval that is far from the $\gamma$-quantile, and hence we say the algorithm mistakes at round $t$ if $\mathcal{M}_t$ occurs. Then, we obtain the following lemma regarding the estimation error of the $\gamma$-quantile:
\begin{lemma}\label{lem:decomp-est-err}
 Let $\tau_t$ be determined by \cref{alg:find-min}. Then, for any random variable $\delta > 0$ that can depend on $x_1,...,x_N$, we have
\begin{align}
  \Mean\bracket*{\abs*{\tilde{x} - \tilde{F}^*(\gamma)}} = \delta + \Mean\bracket*{\max_{t = 1,...,L}\p\cbrace*{\mathcal{M}_t}\ind{\abs*{\tilde{F}^*(\gamma) - \tau_{t}} > \delta}} + 2^{-L}.
\end{align}
\end{lemma}
The concentration inequality gives a bound on the second term in \cref{lem:decomp-est-err}.
\begin{lemma}\label{lem:each-tail}
 Let $z = (z_1,...,z_N)$ be the sanitized version of $(\sign(\tau - x_1),...,\sign(\tau - x_N))$ using the randomized response with the privacy parameter $\epsilon$. If $\gamma > \tilde{F}(\tau)$, 
 \begin{align}
     \p\cbrace*{\Phi(z) > \gamma} \le \exp\paren*{-\frac{(e^\epsilon-1)^2(\tilde{F}(\tau) - \gamma)^2N}{4(e^\epsilon+1)e^{\epsilon}}}.
 \end{align}
 Moreover, if $\gamma < \tilde{F}(\tau)$, 
 \begin{align}
     \p\cbrace*{\Phi(z) < \gamma} \le \exp\paren*{-\frac{(e^\epsilon-1)^2(\tilde{F}(\tau) - \gamma)^2N}{4(e^\epsilon+1)e^{\epsilon}}}.
 \end{align}
\end{lemma}
Choose $\delta$ such that $\tilde{F}^*(2\gamma) - \tilde{F}^*(\gamma) \ge \delta$ or $\tilde{F}^*(\gamma) - \tilde{F}^*(0) \ge \delta$. Then, for any $t$, $\abs*{\tilde{F}(\tau_{t}) - \gamma} \ge \gamma$. Thus, we obtain a bound on the second term in \cref{lem:decomp-est-err} from \cref{lem:each-tail}. We can prove \cref{thm:non-adjust,thm:non-adjust-iid} by deriving bounds on the first term in \cref{eq:decomp} and the first term in \cref{lem:decomp-est-err}, where bounds on these terms can be obtained from \cref{def:fatness}. 

\subsection{Privacy Analysis}

Privacy of \cref{alg:find-min} can be proved easily with application of the sequential composition theorem. We confirm that \cref{alg:find-min} ensures $\epsilon$-local differential privacy. 
\begin{theorem}
 \cref{alg:find-min} is $\epsilon$-locally differentially private.
\end{theorem}
\begin{proof}
 \cref{alg:find-min} uses the randomized response $L$ times with privacy parameter $\epsilon/L$. By the sequential composition theorem of the local differential privacy, the total privacy loss is at most $\epsilon$. 
\end{proof}
Note that \cref{alg:find-min} is $\epsilon$-locally differentially private for any choice of $L$ and $\gamma$. 

\section{Proofs}

\subsection{Proof for Hardness}

We introduce the definition of the differential privacy for proving \cref{thm:hard-worst}. The differential privacy is weaker than the local differential privacy in the sense that any analysis using data satisfying the local differential privacy ensures the differential privacy.  Thus, if the minimum finding problem is difficult under the differential privacy, the problem is also difficult under the local differential privacy. The formal definition of the differential privacy is given as follows:
\begin{definition}[Differential privacy~\citep{Dwork2006}]
 A stochastic mechanism $\mathcal{M}$ mapping from $\dom{X}^N$ to $\dom{Z}$ is $\epsilon$-differentially private if for all $X, X' \in \dom{X}^N$ differing at most one record, and all $S \in \sigma(\dom{Z})$,
 \begin{align}
     \p\cbrace*{\mathcal{M}(X) \in S} \le e^\epsilon\p\cbrace*{\mathcal{M}(X') \in S},
 \end{align}
 where $\sigma(\dom{Z})$ denotes an appropriate $\sigma$-field of $\dom{Z}$.
\end{definition}
Then, we prove \cref{thm:hard-worst}.
\begin{proof}[Proof of \cref{thm:hard-worst}]
 {\bfseries Fixed data case.}
 Let $F_0$ be a cumulative distribution such that $F_0(x) = 0$ for $x \in [-1,1)$. Let $F_1$ be another cumulative distribution such that $F_1(x) > 0$ for any $x \in (-1,1]$. Let $X_0$ and $X_1$ be the users' data generated from $F_0$ and $F_1$, respectively. Then, $X_0$ and $X_1$ have different minimum, whereas the other records are equivalent.
 
 Let $\mathcal{M}$ be a $\epsilon$-differentially private mechanism. Then, its mean absolute errors for $X_0$ and $X_1$ are obtained as 
 \begin{align}
     \Mean\bracket*{\abs*{\mathcal{M}(X_0) - x_{\min}}} =& \Mean\bracket*{\abs*{\mathcal{M}(X_0) - 1}} \\
     \Mean\bracket*{\abs*{\mathcal{M}(X_1) - x_{\min}}} =& \Mean\bracket*{\abs*{\mathcal{M}(X_1) + 1}}. \\
 \end{align}
 Assume 
 \begin{align}
     \Mean\bracket*{\abs*{\mathcal{M}(X_0) - 1}} = o(1).
 \end{align}
 Then, by the Markov inequality, we have
 \begin{align}
     \Mean\bracket*{\abs*{\mathcal{M}(X_0) - 1}} \ge \p\cbrace*{\abs*{\mathcal{M}(X_0) - 1} > 1} = o(1).
 \end{align}
 Because of the differential privacy assumption, we have
 \begin{align}
    \p\cbrace*{\abs*{\mathcal{M}(X_0) - 1} > 1} \ge e^{-\epsilon}\p\cbrace*{\abs*{\mathcal{M}(X_1) - 1} > 1} = o(1).
 \end{align}
 We obtain a lower bound on the error for $X_1$ as
 \begin{align}
     & \Mean\bracket*{\abs*{\mathcal{M}(X_1) + 1}} \\
     \ge& \p\cbrace*{\abs*{\mathcal{M}(X_1) + 1} > 1} \\
     =& 1 - \p\cbrace*{\abs*{\mathcal{M}(X_1) - 1} > 1} = 1 - o(1).
 \end{align}
 This discussion is true even if we exchange $X_0$ and $X_1$. Thus, we obtain the claim.
 
 {\bfseries i.i.d. data case.}
 Let $F_0$ be the same cumulative distribution above. Let $F_1$ be a cumulative distribution such that $F_1(x) = \delta$ for any $x \in (-1,1)$. Note that the distributions of $F_0$ and $F_1$ are supported only on $\cbrace{-1,1}$ such that $\p\cbrace{X = -1} = 0$ under $F_0$ and $\p\cbrace{X = -1} = \delta$ under $F_1$. In the similar manner in the fixed data case, assume 
 \begin{align}
    o(1) = \Mean\bracket*{\abs*{\mathcal{M}(X_1) + 1}} \ge \p\cbrace*{\abs*{\mathcal{M}(X_1) + 1} > 1},
 \end{align}
 where the inequality is obtained by the Markov inequality. Since under $F_0$, all the users' data are $1$, the number of the different records between $X_0$ and $X_1$ follows the binomial distribution with a parameter $N$ and $\delta$. Let $H(X_0,X_1)$ be the number of the different records between $X_0$ and $X_1$. Then, from the differential privacy assumption, we have
 \begin{align}
    \p\cbrace*{\abs*{\mathcal{M}(X_1) + 1} > 1} \ge& \Mean\bracket*{e^{-\epsilon H(X_0,X_1)}\p\cbrace*{\abs*{\mathcal{M}(X_0) + 1} > 1 | X_0}} \\
    =& \Mean\bracket*{e^{-\epsilon H(X_0,X_1)}}\p\cbrace*{\abs*{\mathcal{M}(X_0) + 1} > 1} \\
    \ge& (1-\delta)^N\p\cbrace*{\abs*{\mathcal{M}(X_0) + 1} > 1}.
 \end{align}
 For $\delta = o(1/N)$, we have 
 \begin{align}
    \p\cbrace*{\abs*{\mathcal{M}(X_1) + 1} > 1} \ge  (1 - o(1))\p\cbrace*{\abs*{\mathcal{M}(X_0) + 1} > 1}.
 \end{align}
 In the same manner of the fixed data case, we get the claim.
\end{proof}

\subsection{Proof for Upper Bounds}

\begin{proof}[Proof of \cref{lem:decomp-est-err}]
Let $t_1 < t_2 < ... < t_M$ be the rounds that the algorithm mistakes. By the definition of $\mathcal{M}_t$, we have
\begin{align}
    \tilde{F}^*(\gamma) \le \tau_{t_1} \le \tau_{t_2} \le ... \le \tau_{t_M},
\end{align}
or
\begin{align}
    \tilde{F}^*(\gamma) \ge \tau_{t_1} \ge \tau_{t_2} \ge ... \ge \tau_{t_M},
\end{align}
Since the algorithm does not mistake after $t_M$ round, we have $\abs*{t_M - \tilde{x}} \le 2^{-L}$. Let $t_\delta$ be the maximum round $t$ such that $\abs*{\tilde{F}^*(\gamma) - \tau_t} \le \delta$. We remark that $t_\delta$ is the random variable over $[L]$. Then, we have
\begin{align}
  \Mean\bracket*{\abs*{\tilde{x} - \tilde{F}^*(\gamma)}} = \Mean\bracket*{\abs*{\tilde{F}^*(\gamma) - \tau_{t_\delta}}} + \Mean\bracket*{\abs*{\tau_{t_\delta} - \tau_{t_M}}} + 2^{-L}.
\end{align}
 Since the difference between $\tau_t$ and $\tau_{t+1}$ is $2^{-t}$, we have
 \begin{align}
    \abs*{\tau_{t_\delta} - \tau_{t_M}} \le& \sum_{t = t_\delta + 1}^L \ind{\mathcal{M}_t}2^{-t} \le \sum_{t = 1}^L \ind{\mathcal{M}_t, \abs*{\tilde{F}^*(\gamma) - \tau_t} > \delta}2^{-t}.
 \end{align}
 Hence,
 \begin{align}
    \Mean\bracket*{\abs*{\tau_{t_m} - \tau_{t_M}}} \le& \Mean\bracket*{\sum_{t = 1}^L \p\cbrace*{\mathcal{M}_t}\ind{\abs*{\tilde{F}^*(\gamma) - \tau_t} > \delta}2^{-t}} \\
    \le& \Mean\bracket*{\max_{t=1,...,L} \p\cbrace*{\mathcal{M}_t}\ind{\abs*{\tilde{F}^*(\gamma) - \tau_t} > \delta}\sum_{t=1}^L2^{-t}} \\
    \le& \Mean\bracket*{\max_{t=1,...,L}\p\cbrace*{\mathcal{M}_t}\ind{\abs*{\tilde{F}^*(\gamma) - \tau_t} > \delta}}.
 \end{align}
\end{proof}

\begin{proof}[Proof of \cref{thm:non-adjust}]
 From \cref{eq:decomp,lem:decomp-est-err,lem:each-tail}, with an appropriate $\delta$, we have
 \begin{multline}
     \mathrm{Err} \le \Mean\bracket*{\abs*{\tilde{F}^*(\gamma) - x_{\min}}} + \max\cbrace*{\Mean\bracket*{\abs*{\tilde{F}^*(\gamma) - \tilde{F}^*(0)}},\Mean\bracket*{\abs*{\tilde{F}^*(\gamma) - \tilde{F}^*(2\gamma)}}} \\ + \exp\paren*{-\frac{(e^{\epsilon/L}-1)^2\gamma^2N}{4(e^{\epsilon/L}+1)e^{\epsilon/L}}} + 2^{-L}, \label{eq:error-last}
 \end{multline}
 where we use the fact that $\tilde{F}^*$ is non-decreasing. The sum of the first two terms is bounded above by
 \begin{align}
    2\Mean\bracket*{\abs*{\tilde{F}^*(2\gamma) - x_{(1)}}}. \label{eq:quantile-error}
 \end{align}
 If $2\gamma < C_1(C_2 - x_{(1)})^\alpha$, we have $\tilde{F}^*(2\gamma) \in (x_{(1)},C_2)$. Hence, under $\alpha$-fatness, we have $\abs*{\tilde{F}^*(2\gamma) - x_{(1)}} \le \paren*{\frac{2\gamma}{C_1}}^{1/\alpha}$.  Substituting this into \cref{eq:error-last} yields the desired result. 
\end{proof}

\begin{proof}[Proof of \cref{thm:non-adjust-iid}]
 The proof follows the same manner of that of \cref{thm:non-adjust} except a bound on \cref{eq:quantile-error}. Let $U_{(1)},...,U_{(N)}$ be the order statistics of the uniform distribution on $[0,1]$. Then, we have
 \begin{align}
     x_{(k)} = F^*(U_{(k)}).
 \end{align}
 Hence,
 \begin{align}
     & \Mean\bracket*{\abs*{\tilde{F}^*(2\gamma) - x_{\min}}} \\
     =& \Mean\bracket*{\abs*{F^*(U_{(\ceil{2\gamma N})} - F^*(0)}} \\
     \le& \frac{1}{C_1^{1/\alpha}} \Mean\bracket*{U_{(\ceil{2\gamma N})}^{1/\alpha}}.
 \end{align}
 Since $U_{(k)}$ follows the beta distribution with parameters $k$ and $N-k+1$, we have
 \begin{align}
     \Mean\bracket*{U_{(k)}^{1/\alpha}} =& \frac{\Beta(k+\frac{1}{\alpha},N-k+1)}{\Beta(k,N-k+1)} \\
     =& \frac{\Gamma(k+\frac{1}{\alpha})\Gamma(N+1)}{\Gamma(N+1+\frac{1}{\alpha})\Gamma(k)} = \frac{(k)^{\overline{1/\alpha}}}{(N+1)^{\overline{1/\alpha}}}.
 \end{align}
\end{proof}

\begin{proof}[Proof of \cref{lem:each-tail}]
 We use the concentration inequality from \citep{Boucheron2003ConcentrationMethod}. Let $Z = f(z_1,...,z_N) = \frac{e^\epsilon+1}{e^\epsilon-1}\sum_{i=1}^Nz_i$. Let $Z^{(i)} = f(z_1,...,z_{i-1},z'_i,z_{i+1}...,z_N)$, where $z'_i$ be the independent copy of $z_i$. Define
 \begin{align}
     V_+ =& \Mean\bracket*{ \sum_{i=1}^N\paren*{Z - Z^{(i)}}^2\ind{Z > Z^{(i)}} | z_1,...,z_N} \\
      =& \frac{4(e^\epsilon+1)^2}{(e^\epsilon-1)^2}\sum_{i=1}^N\p\cbrace*{z'_i = -1}\ind{z_i = 1}.
 \end{align}
 Moreover, we have
 \begin{align}
     V_- =& \Mean\bracket*{ \sum_{i=1}^N\paren*{Z - Z^{(i)}}^2\ind{Z < Z^{(i)}} | z_1,...,z_N} \\
     =& \frac{4(e^\epsilon+1)^2}{(e^\epsilon-1)^2}\sum_{i=1}^N\p\cbrace*{z'_i = 1}\ind{z_i = -1}.
 \end{align}
 From Theorem 2 in \citep{Boucheron2003ConcentrationMethod}, for $\theta > 0$ and $\lambda \in (0,1/\theta)$, we have 
 \begin{align}
     \ln\Mean\bracket*{e^{\lambda\paren*{Z - \Mean[Z]}}} \le \frac{\lambda\theta}{1-\lambda\theta}\ln\Mean\bracket*{e^{\frac{\lambda V_+}{\theta}}}, \label{eq:thm2-1}
 \end{align}
 and
 \begin{align}
     \ln\Mean\bracket*{e^{-\lambda\paren*{Z - \Mean[Z]}}} \le \frac{\lambda\theta}{1-\lambda\theta}\ln\Mean\bracket*{e^{\frac{\lambda V_-}{\theta}}}. \label{eq:thm2-2}
 \end{align}
 
 By definition, we have
 \begin{align}
     & \frac{\lambda\theta}{1-\lambda\theta}\ln\Mean\bracket*{e^{\frac{\lambda V_+}{\theta}}}  \\
     =& \frac{\lambda\theta}{1-\lambda\theta}\sum_{i=1}^N\ln\paren*{\p\cbrace*{z_i = 1}e^{\frac{4\lambda}{\theta}\frac{(e^\epsilon+1)^2}{(e^\epsilon-1)^2}\p\cbrace*{z_i = -1}}} \\
     =& \begin{multlined}[t]
     \frac{4\lambda^2}{1-\lambda\theta}\sum_{i=1}^N\frac{(e^\epsilon+1)^2}{(e^\epsilon-1)^2}\p\cbrace*{z_i = -1} \\ + \frac{\lambda\theta}{1-\lambda\theta}\sum_{i=1}^N\ln\p\cbrace*{z_i = 1}.
     \end{multlined}
 \end{align}
 As $\theta \to 0$, we obtain
 \begin{align}
     & \lim_{\theta \to 0}\frac{\lambda\theta}{1-\lambda\theta}\ln\Mean\bracket*{e^{\frac{\lambda V_+}{\theta}}} \\
     =& 4\lambda^2\frac{(e^\epsilon+1)^2}{(e^\epsilon-1)^2}\sum_{i=1}^N\p\cbrace*{z_i = -1} \\
     \le& 4\lambda^2\frac{(e^\epsilon + 1)e^\epsilon N}{(e^\epsilon-1)^2} .
 \end{align}
 In the similar way, we obtain
 \begin{align}
     & \lim_{\theta \to 0}\frac{\lambda\theta}{1-\lambda\theta}\ln\Mean\bracket*{e^{\frac{\lambda V_-}{\theta}}} \\
     =& 4\lambda^2\frac{(e^\epsilon+1)^2}{(e^\epsilon-1)^2}\sum_{i=1}^N\p\cbrace*{z_i = 1} \\
     \le& 4\lambda^2\frac{(e^\epsilon + 1)e^\epsilon N}{(e^\epsilon-1)^2}.
 \end{align}
 From the Markov inequality, we have 
 \begin{align}
     \p\cbrace*{Z > \Mean[Z] + t} \le \frac{e^{\lambda(Z - \Mean[Z])}}{e^{\lambda t}},
 \end{align}
 and
 \begin{align}
     \p\cbrace*{Z < \Mean[Z] - t} \le \frac{e^{-\lambda(Z - \Mean[Z])}}{e^{\lambda t}}.
 \end{align}
 Optimizing $\lambda$ gives that
 \begin{align}
     \p\cbrace*{Z > \Mean[Z] + t} \le \exp\paren*{-\frac{(e^\epsilon-1)^2t^2}{16(e^\epsilon+1)e^{\epsilon}N}},
 \end{align}
 and
 \begin{align}
     \p\cbrace*{Z < \Mean[Z] - t} \le \exp\paren*{-\frac{(e^\epsilon-1)^2t^2}{16(e^\epsilon+1)e^{\epsilon}N}}.
 \end{align}
 Noting that
 \begin{align}
     & \p\cbrace*{Z > \Mean[Z] + t} \\
     =& \p\cbrace*{\frac{e^\epsilon+1}{e^\epsilon-1}\sum_{i=1}^Nz_i > 2N\tilde{F}(\tau) + t}.
 \end{align}
 Thus, setting $t = 2N (\gamma - \tilde{F}(\tau))$ yields the desired claim.
\end{proof}

\subsection{Proof for Lower Bound}

\begin{proof}[Proof of \cref{thm:fat-worst}]
 {\bfseries i.i.d. data case.} We use the lower bound from \citet{Duchi2016} for the i.i.d. case. 
 \begin{theorem}[\citep{Duchi2016}]\label{thm:duchi-lower}
   Given $\delta > 0$, let $F$ and $F'$ be the cumulative functions such that these minimums, denoted as $x_{\min}$ and $x'_{\min}$, respectively, differs at least $2\delta$, i.e., $\abs*{x_{\min} - x'_{\min}} \ge 2\delta$. For $\epsilon \in [0,22/35]$, for any $\epsilon$-locally differentially private mechanism, there exists a cumulative function $F_0$ such that the error under $F_0$ is lower bounded as
   \begin{align}
       \mathrm{Err} \ge \abs*{\delta}\paren*{\frac{1}{2} - \sqrt{N\epsilon^2\TV(F,F')^2}},
   \end{align}
   where $\TV$ denotes the total variation distance.
 \end{theorem}
 From \cref{thm:duchi-lower}, we can obtain a lower bound by designing $F$ and $F'$ so that $\TV(F,F')$ is minimized while satifying $\abs*{x_{\min} - x'_{\min}} \ge 2\delta$ simultaneously. We select different choices of $F$ and $F'$ for $\alpha \in (0,1)$ and $\alpha \ge 1$.
 
 {\bfseries Case $\alpha \in (0,1)$.}
 Set 
 \begin{align}
    F(x) =& \begin{dcases}
      (x + 1)^\alpha & \textif x \in [-1,0] \\
      1 & \otherwise,
    \end{dcases} \\
    F'(x) =& \begin{dcases}
      0 & \textif x \in [-1,-1+2\delta) \\
      (x + 1 - 2\delta)^\alpha & \textif x \in [-1+2\delta,2\delta] \\
      1 & \otherwise.
    \end{dcases}
 \end{align}
 Then, the total variation distance between $F$ and $F'$ is obtained as
 \begin{align}
     \TV(F,F') =& 1 - \alpha\int_{-1+2\delta}^{0} \min\cbrace*{(x + 1)^{\alpha-1},(x+1-2\delta)^{\alpha-1}} dx \\
     =& 1 - \alpha\int_{2\delta}^1 x^{\alpha-1} dx \\
     =& 1 - (1 - (2\delta)^\alpha) = (2\delta)^\alpha.
 \end{align}
 Hence, setting $\delta = (16\epsilon^2N)^{-1/2\alpha}/2$ yields that 
 \begin{align}
    \mathrm{Err} \ge \frac{1}{8}\paren*{\frac{1}{16\epsilon^2N}}^{1/2\alpha}.
 \end{align}
 
 {\bfseries Case $\alpha \ge 1$.}
 Set 
 \begin{align}
    F(x) =& \begin{dcases}
      (x + 1)^\alpha & \textif x \in [-1,0] \\
      1 & \otherwise,
    \end{dcases} \\
    F'(x) =& \begin{dcases}
      0 & \textif x \in [-1,-1+2\delta) \\
      (x + 1)^\alpha - (2\delta)^\alpha & \textif x \in [-1+2\delta,0] \\
      1 & \otherwise.
    \end{dcases}
 \end{align}
 Then, the total variation distance between $F$ and $F'$ is obtained as
 \begin{align}
     \TV(F,F') =& \int_0^{2\delta} \alpha x^{\alpha - 1} dx \\
     =& (2\delta)^{\alpha}.
 \end{align}
 Hence, with the same setting of $\delta$ for $\alpha \in (0,1)$ case yields the same lower bound.
\end{proof}

\section{More Experiments: Synthetic Datasets}
The full experimental results on the synthetic dataset can be found in \url{expr.pdf}.

\section{More Experiments: Purchase History Data}

The rest of the experimental results on the purchase history dataset is shown in here.

\begin{figure}[H]
 \centering
 \begin{subfigure}{.4\textwidth}
   \includegraphics[width=\textwidth]{real/age-vs-money/filtered-result_0-10000.pdf}
   \caption{{[\yen 0, \yen 10,000], $\alpha=3.98$}}
 \end{subfigure}
 \begin{subfigure}{.4\textwidth}
   \includegraphics[width=\textwidth]{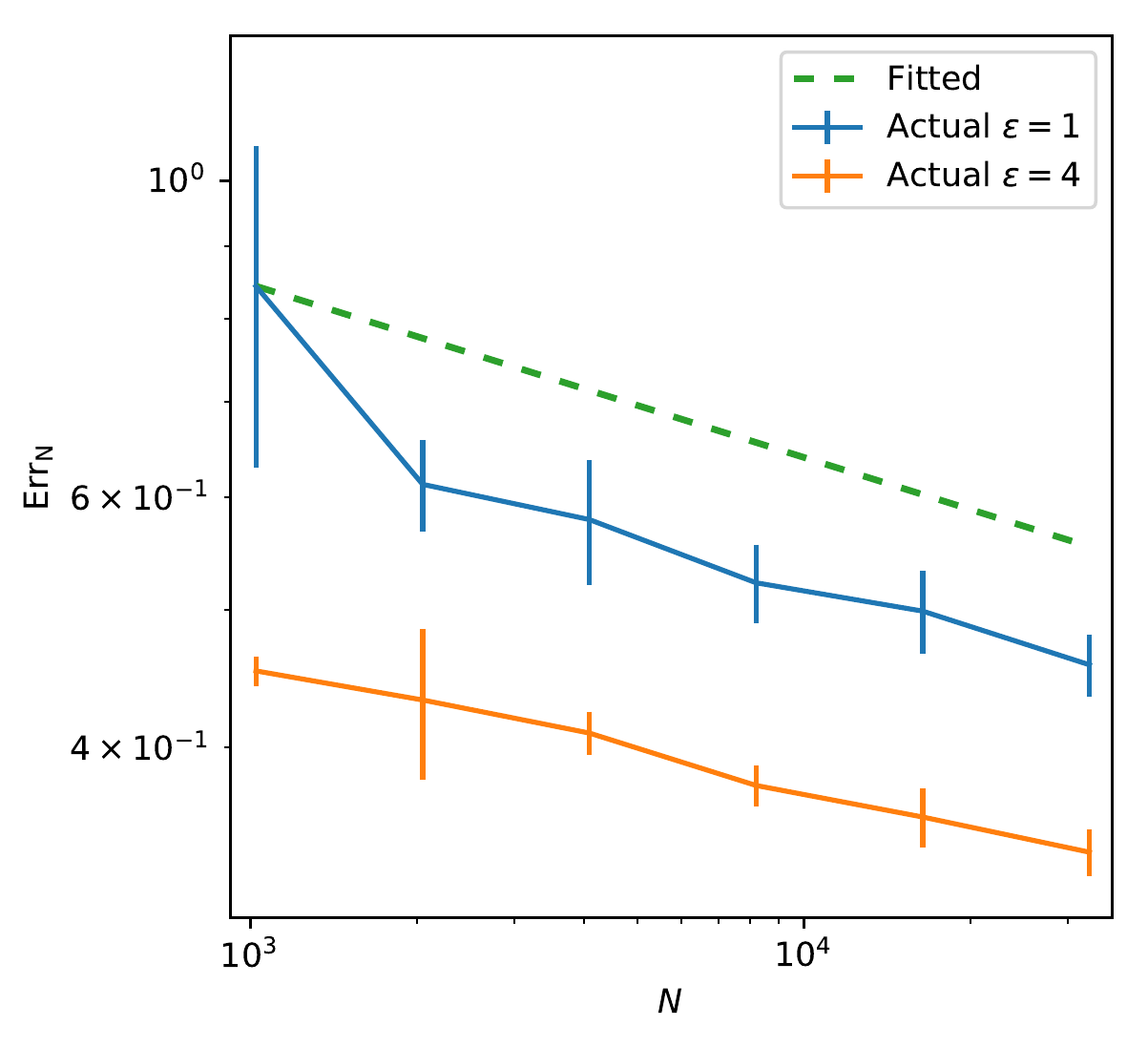}
   \caption{{[\yen 10,000, \yen 20,000], $\alpha=4.11$}}
 \end{subfigure}
 \begin{subfigure}{.4\textwidth}
   \includegraphics[width=\textwidth]{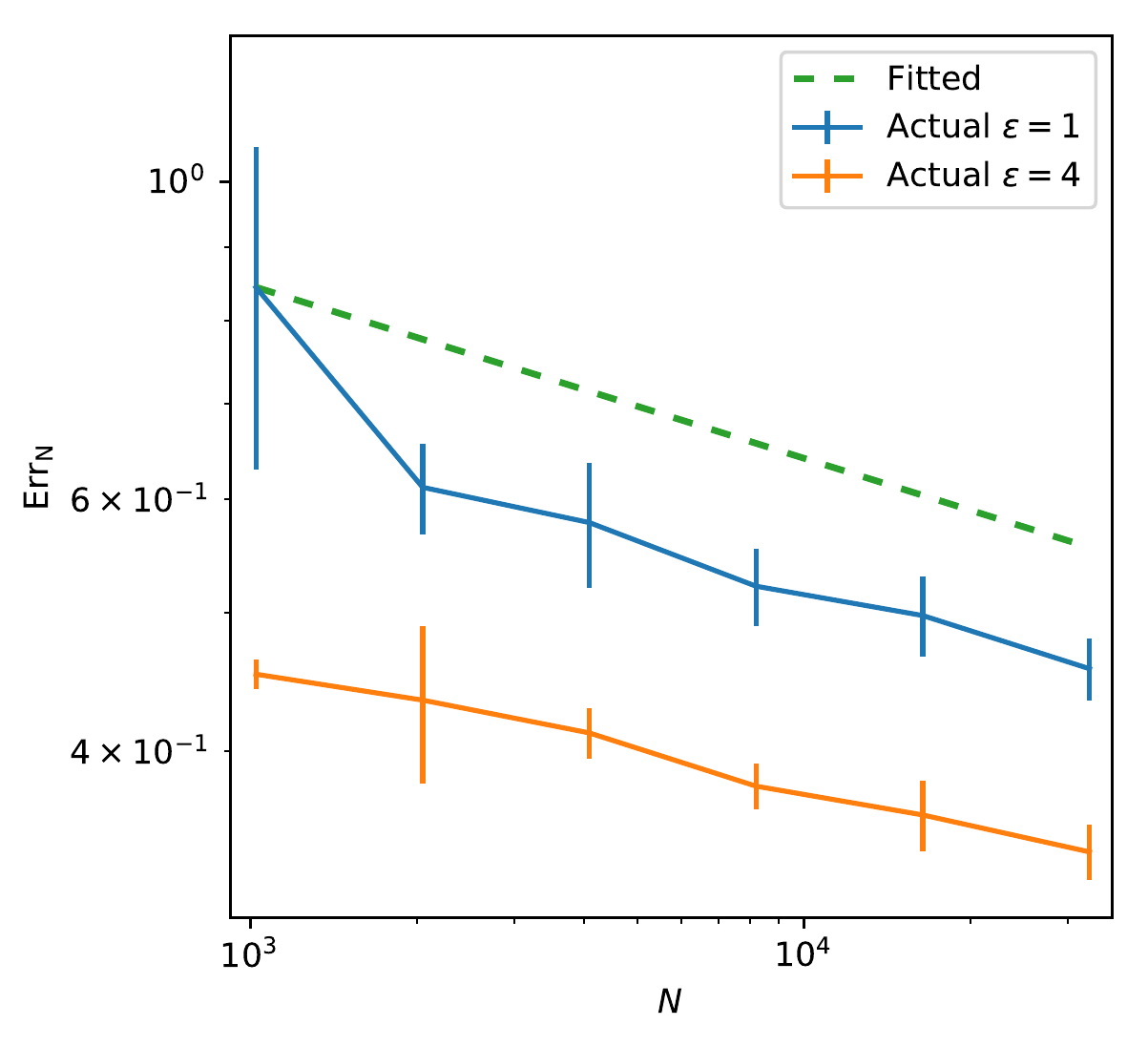}
   \caption{{[\yen 20,000, \yen 30,000], $\alpha=4.14$}}
 \end{subfigure}
 \begin{subfigure}{.4\textwidth}
   \includegraphics[width=\textwidth]{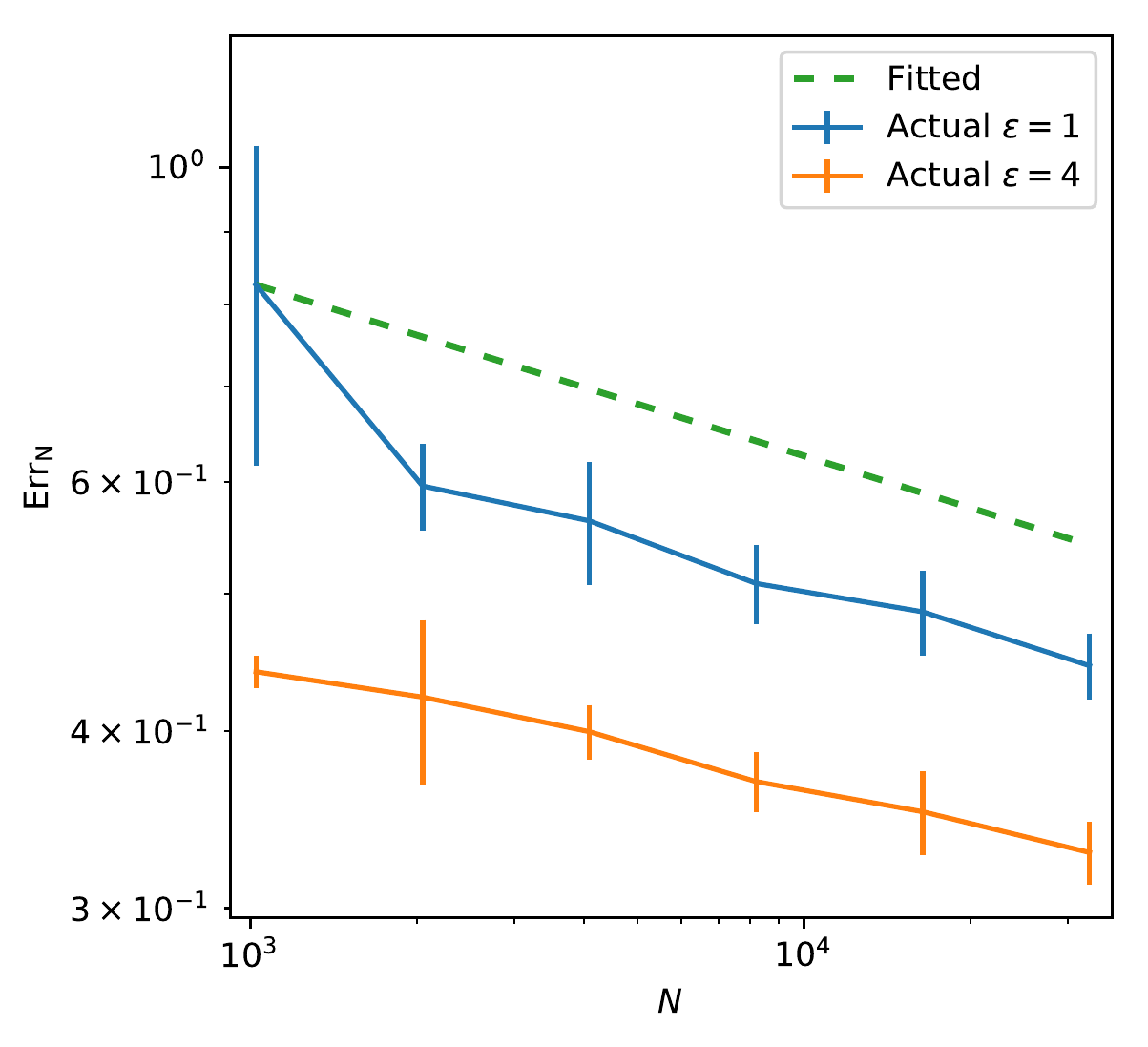}
   \caption{{[\yen 30,000, \yen 40,000], $\alpha=4.10$}}
 \end{subfigure}
 \begin{subfigure}{.4\textwidth}
   \includegraphics[width=\textwidth]{real/age-vs-money/filtered-result_40000-50000.pdf}
   \caption{{[\yen 40,000, \yen 50,000], $\alpha=4.31$}}
 \end{subfigure}
 \begin{subfigure}{.4\textwidth}
   \includegraphics[width=\textwidth]{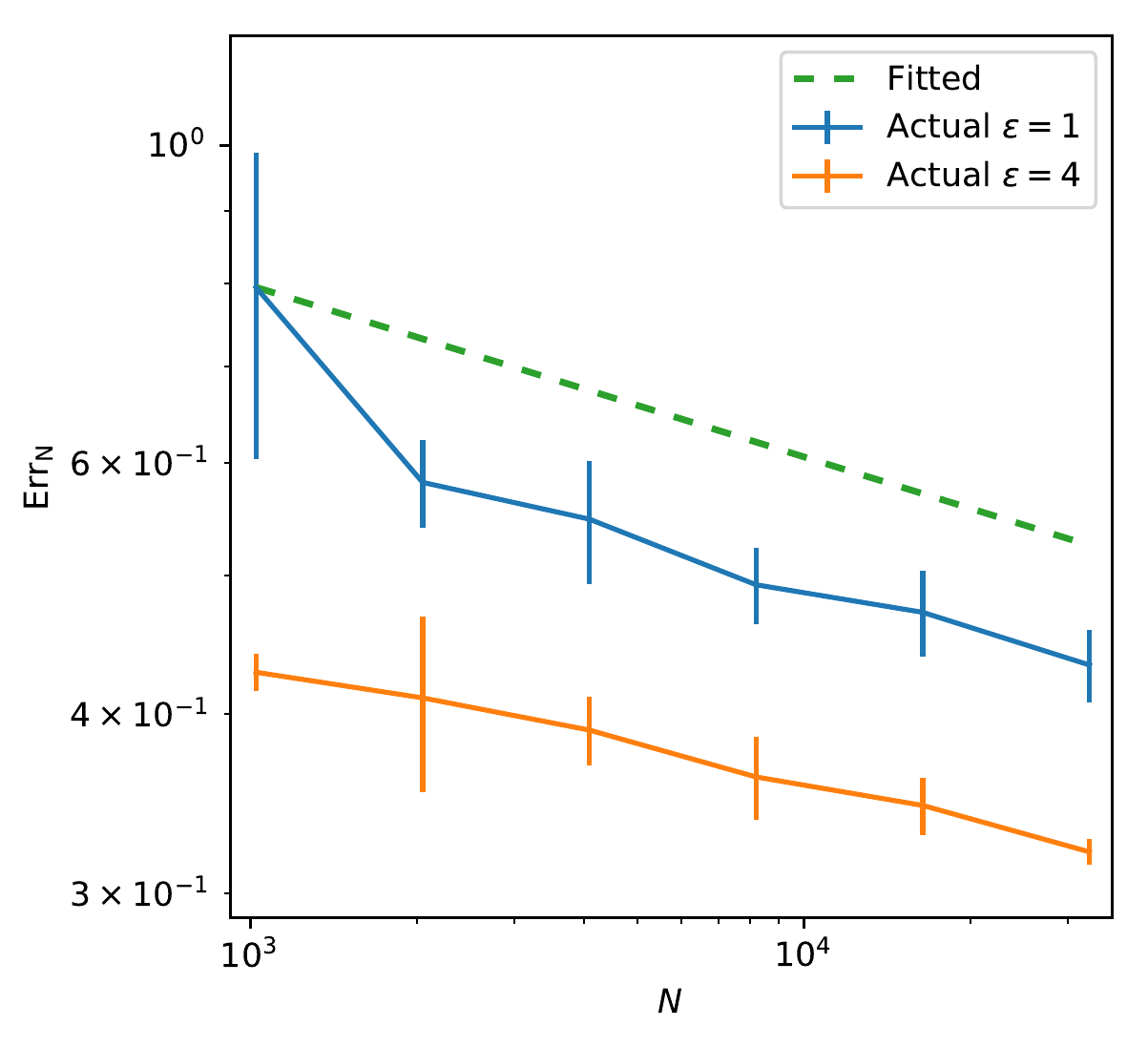}
   \caption{{[\yen 50,000, \yen 60,000], $\alpha=4.17$}}
 \end{subfigure}
 \caption{Experimental results for {\bfseries Task1}.}\label{fig:real-task1}
\end{figure}

\begin{figure}[H]
 \centering
 \begin{subfigure}{.4\textwidth}
   \includegraphics[width=\textwidth]{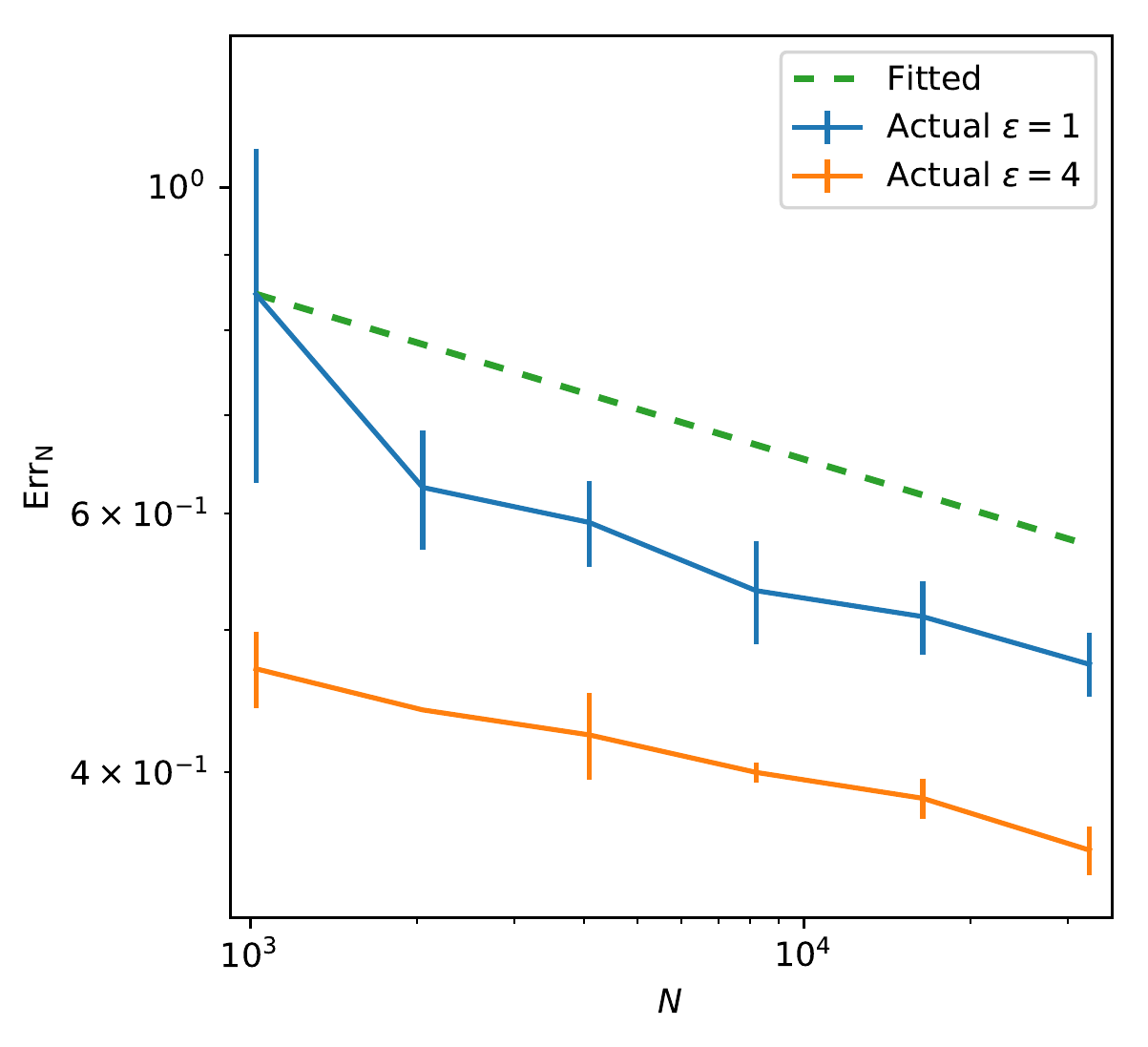}
   \caption{{audio-visual-equipment-and-camera, $\alpha=4.40$}}
 \end{subfigure}
 \begin{subfigure}{.4\textwidth}
   \includegraphics[width=\textwidth]{real/age-vs-category/filtered-result_baby_kids_maternity.pdf}
   \caption{{baby-kids-maternity, $\alpha=4.79$}}
 \end{subfigure}
 \begin{subfigure}{.4\textwidth}
   \includegraphics[width=\textwidth]{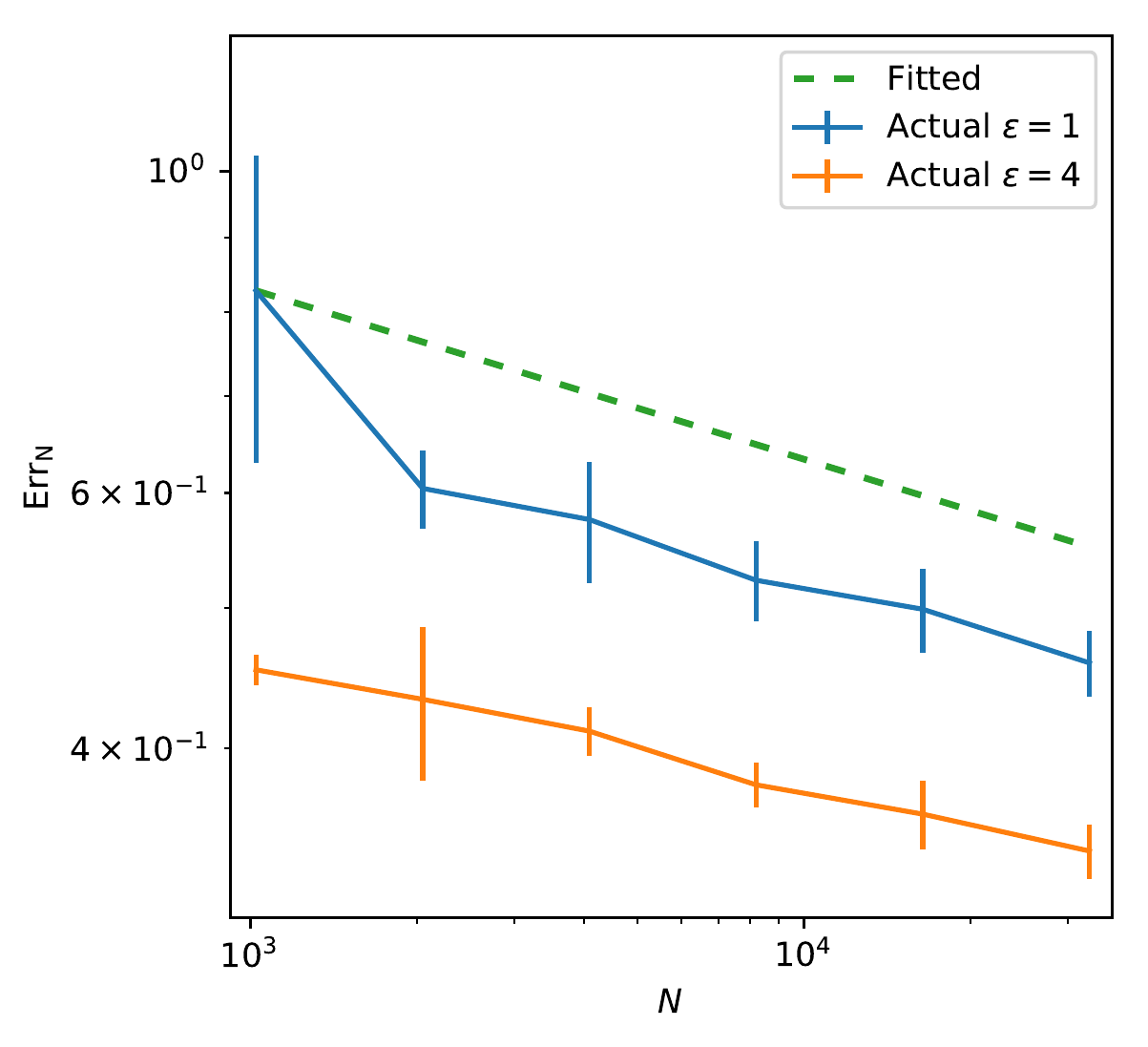}
   \caption{{bicycle-car-motorcycle, $\alpha=4.26$}}
 \end{subfigure}
 \begin{subfigure}{.4\textwidth}
   \includegraphics[width=\textwidth]{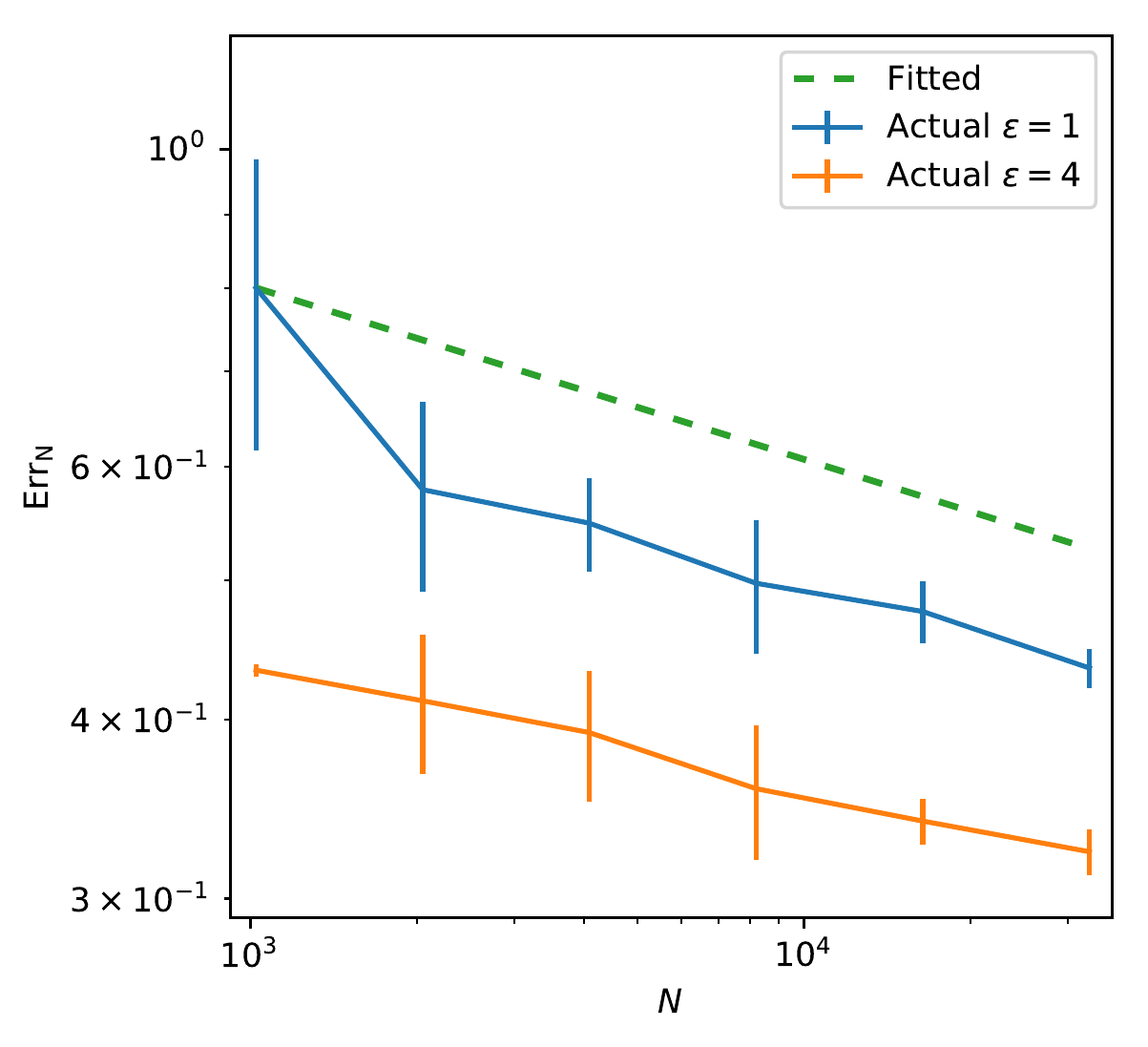}
   \caption{{book-magazine-comic, $\alpha=4.14$}}
 \end{subfigure}
 \begin{subfigure}{.4\textwidth}
   \includegraphics[width=\textwidth]{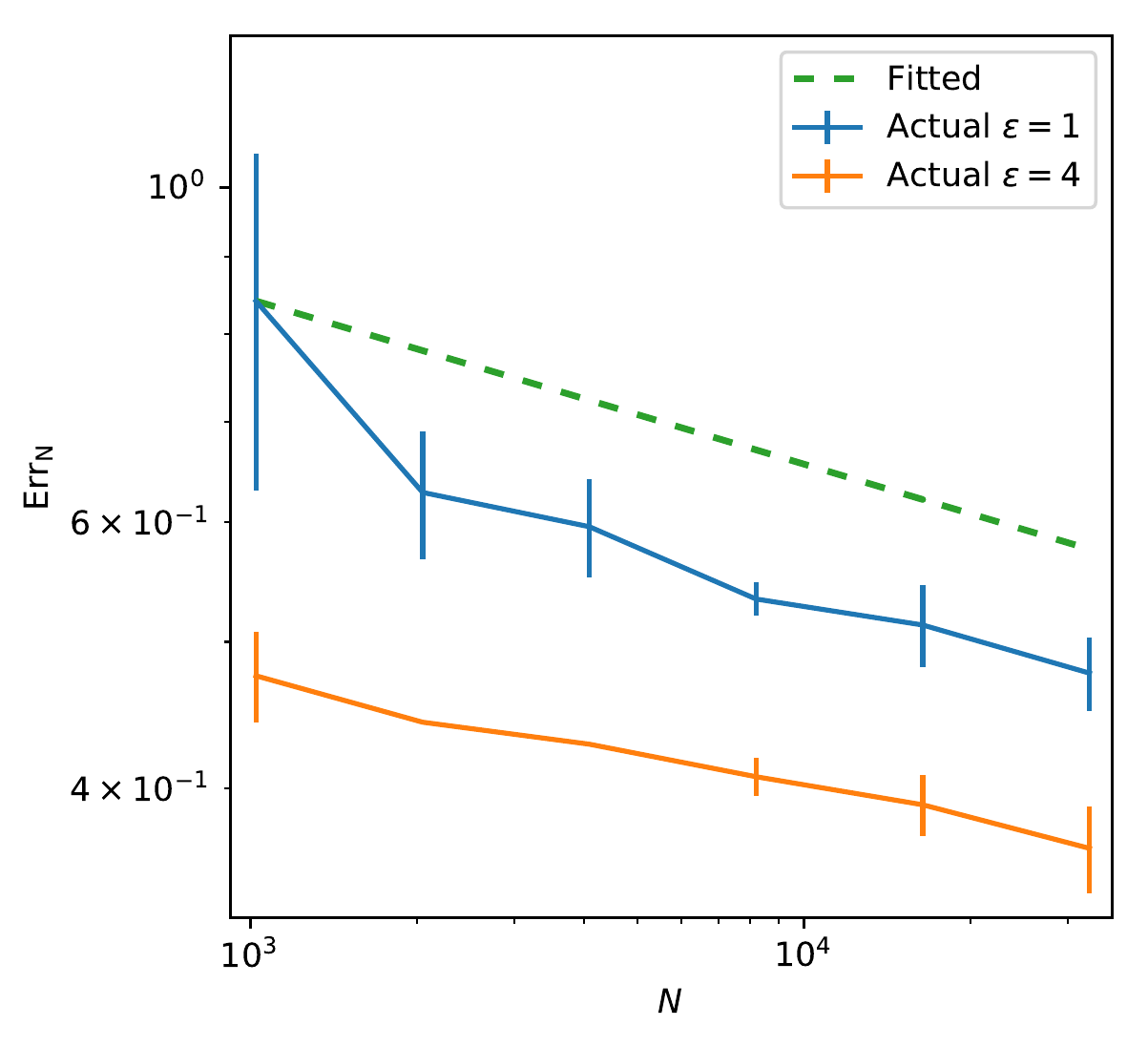}
   \caption{{consumer-electronics, $\alpha=4.58$}}
 \end{subfigure}
 \begin{subfigure}{.4\textwidth}
   \includegraphics[width=\textwidth]{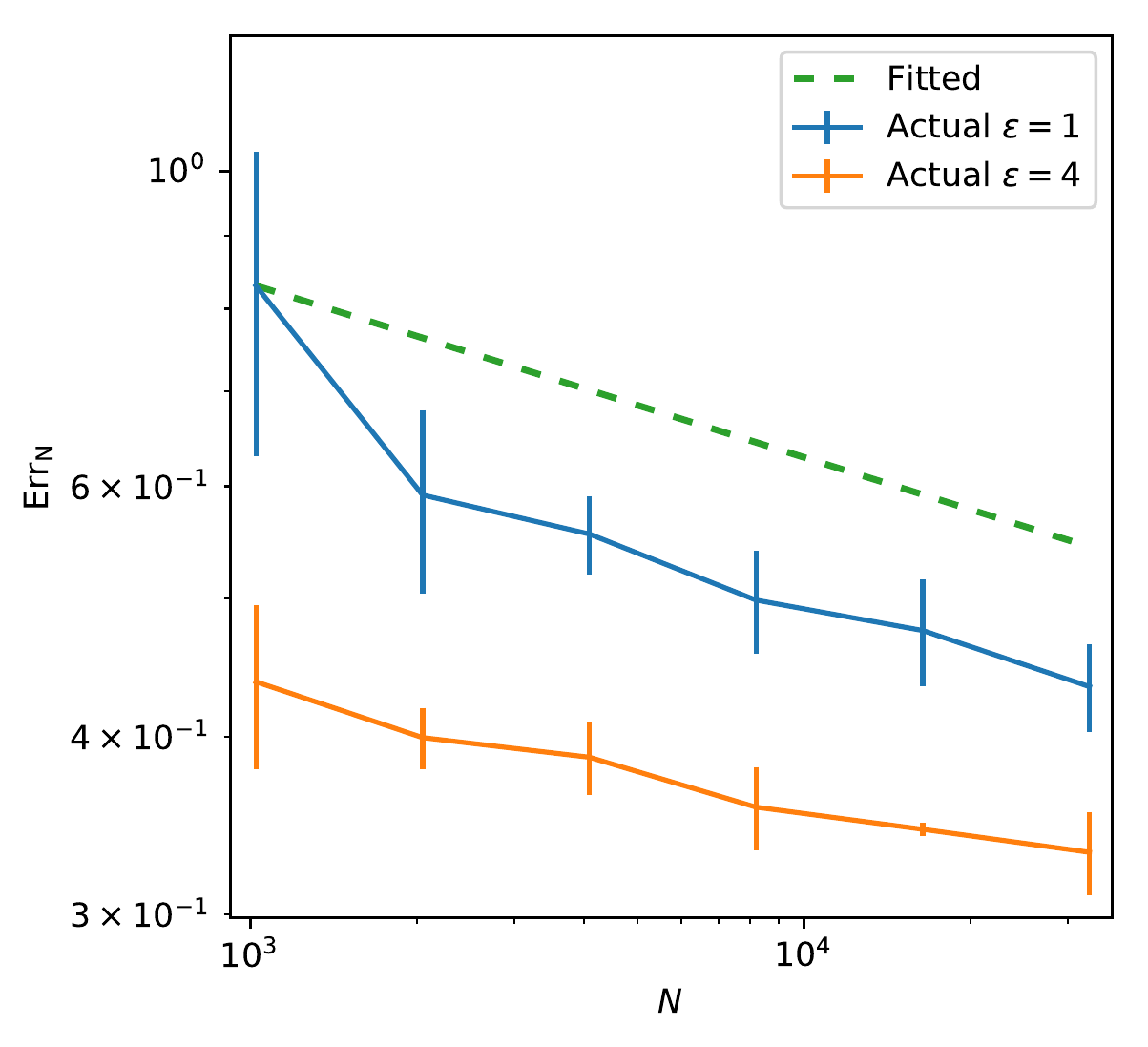}
   \caption{{cosmetics-and-fragrance, $\alpha=4.10$}}
 \end{subfigure}
\caption{Experimental results for {\bfseries Task2}, Part I.}\label{fig:real-task21}
\end{figure}

\begin{figure}[H]
 \centering
 \begin{subfigure}{.4\textwidth}
   \includegraphics[width=\textwidth]{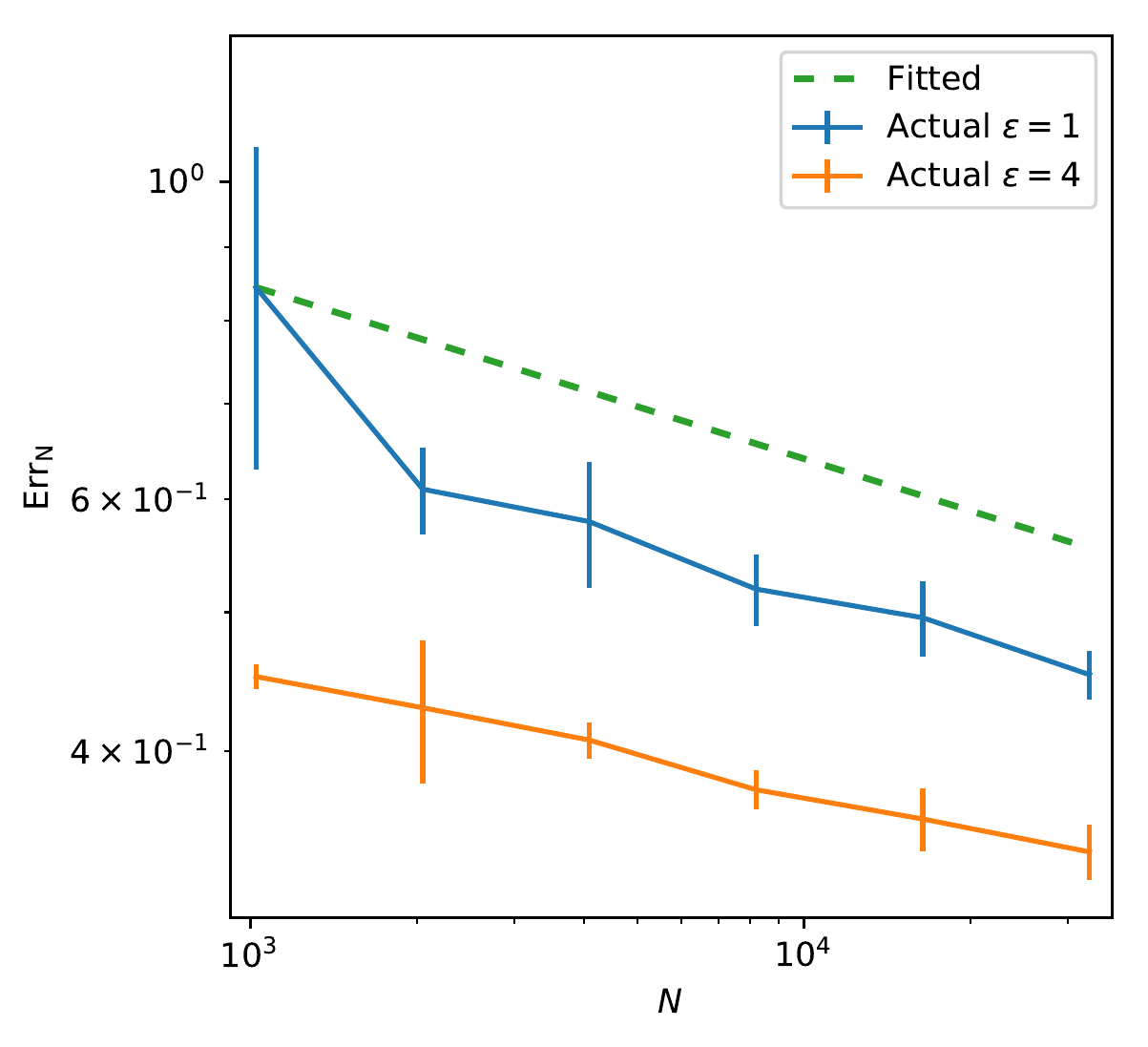}
   \caption{{diet-and-health, $\alpha=4.13$}}
 \end{subfigure}
 \begin{subfigure}{.4\textwidth}
   \includegraphics[width=\textwidth]{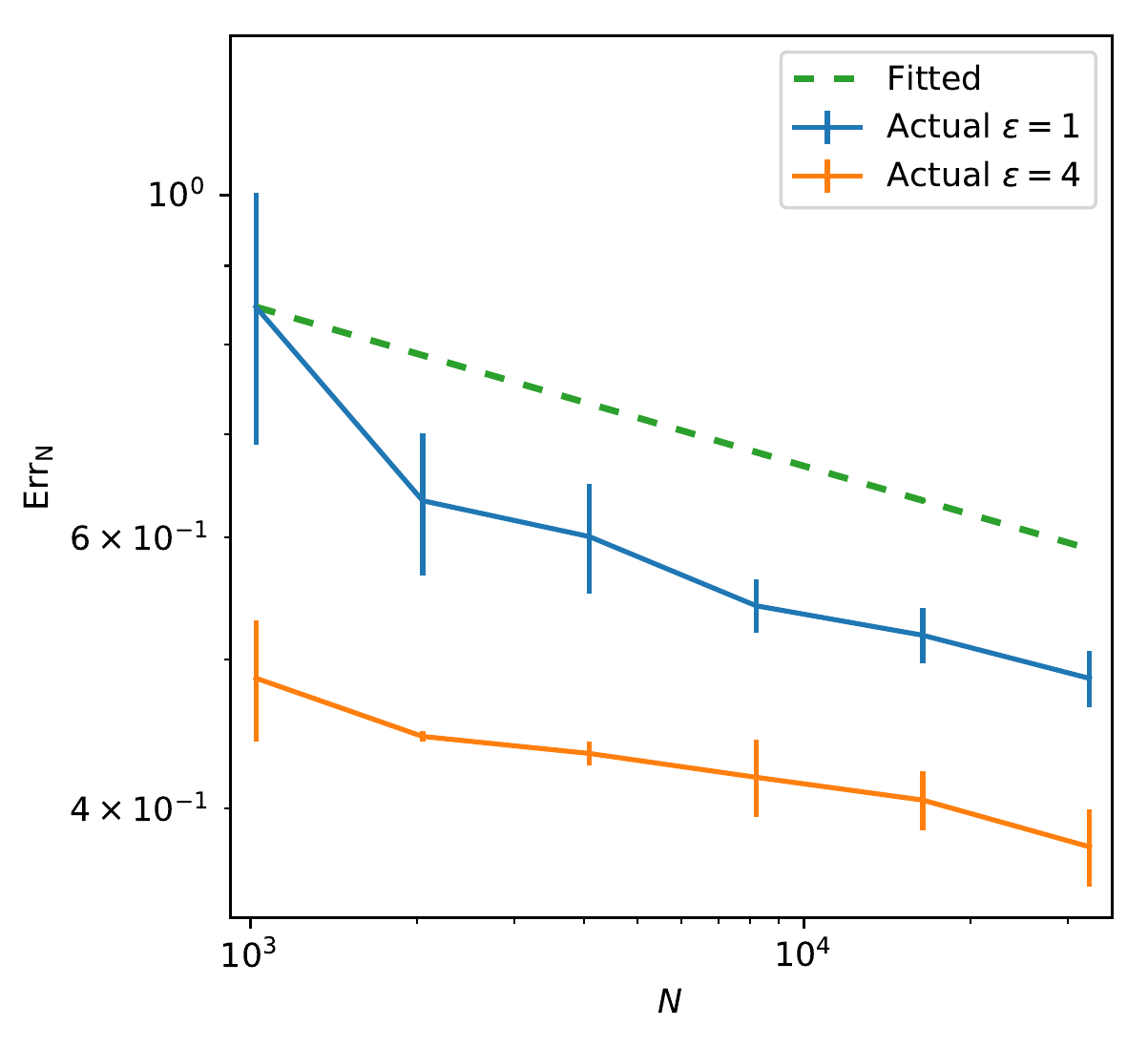}
   \caption{{diy-tool-stationery, $\alpha=4.79$}}
 \end{subfigure}
 \begin{subfigure}{.4\textwidth}
   \includegraphics[width=\textwidth]{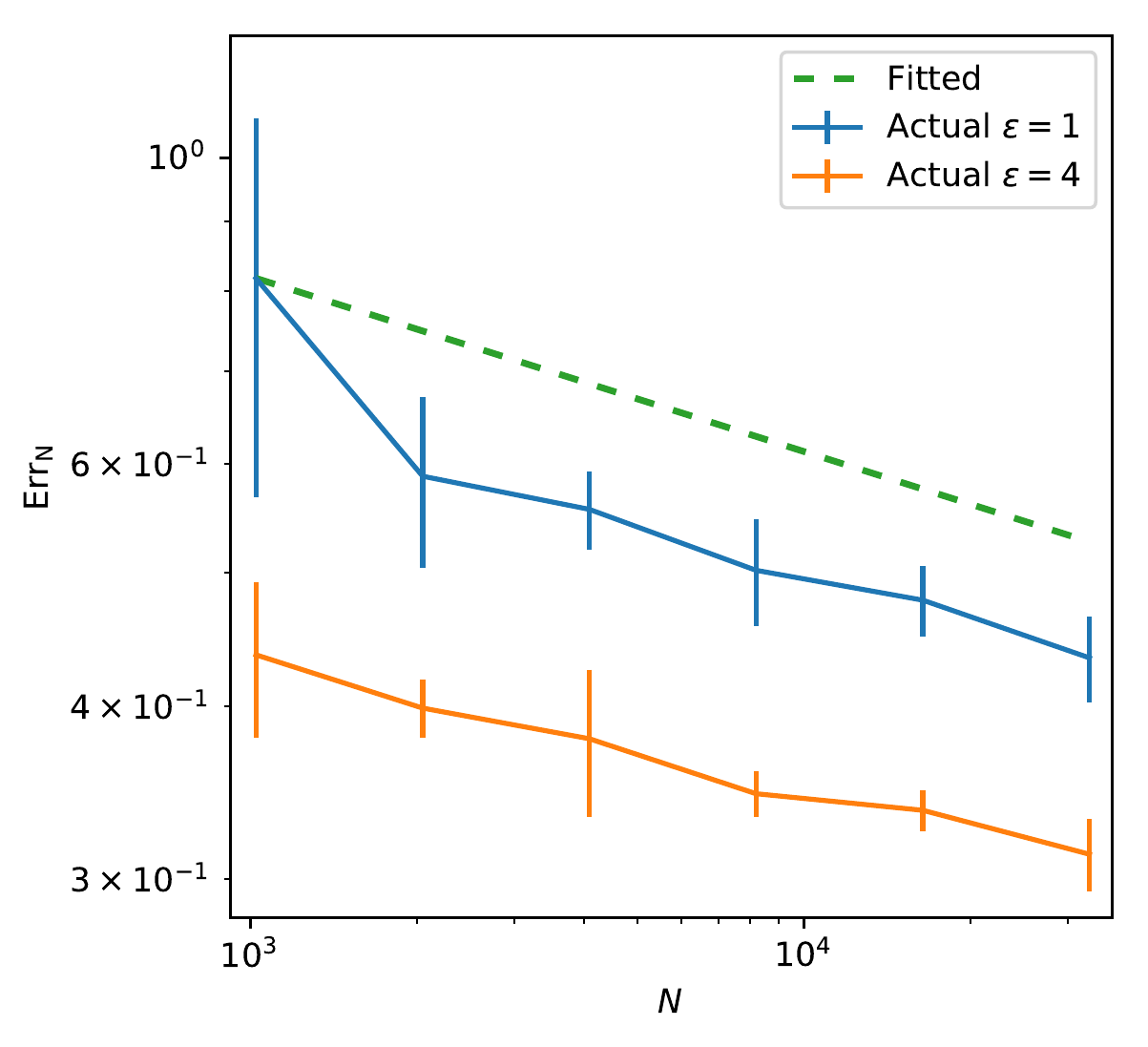}
   \caption{{fashion, $\alpha=3.938$}}
 \end{subfigure}
 \begin{subfigure}{.4\textwidth}
   \includegraphics[width=\textwidth]{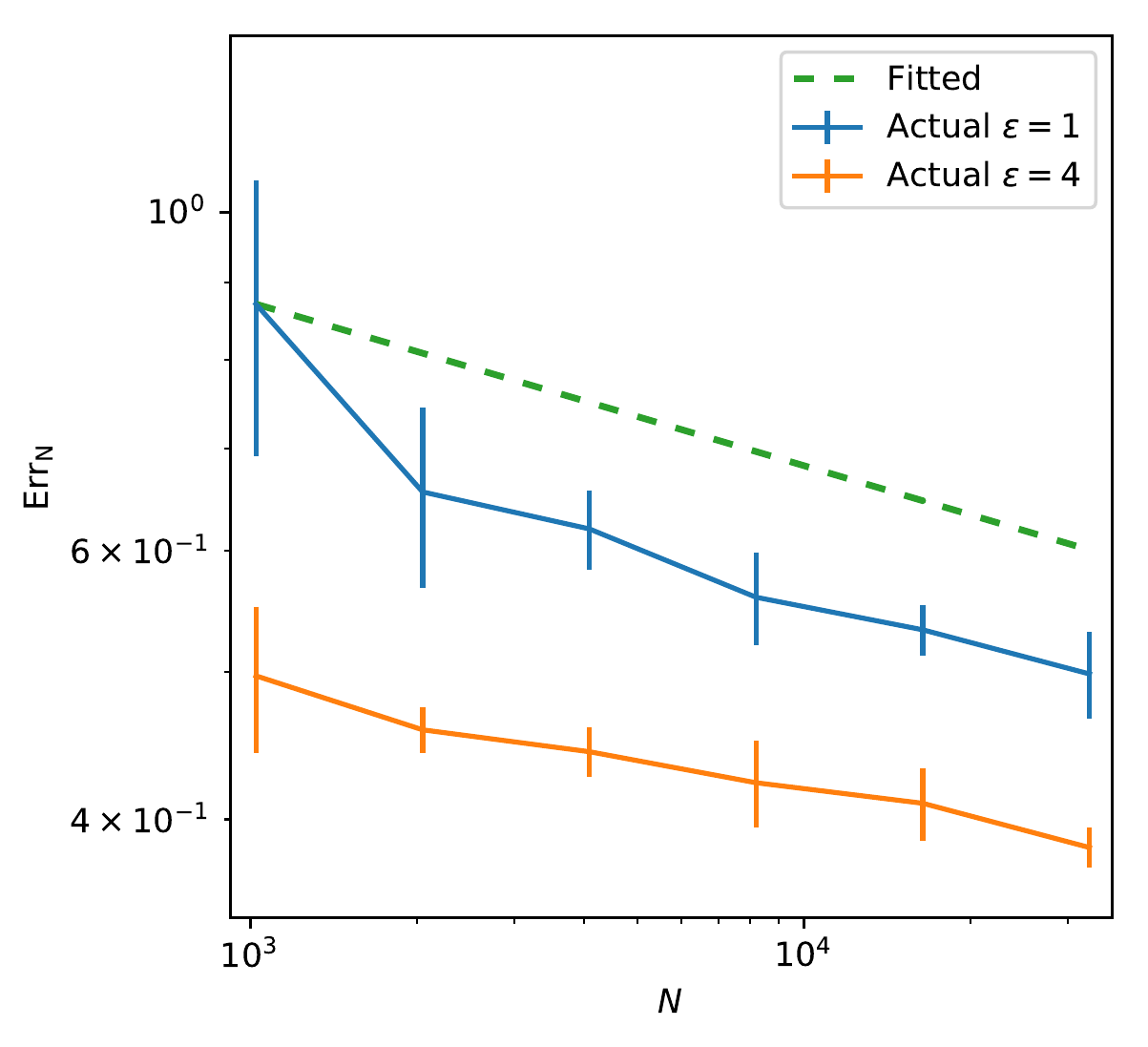}
   \caption{{food, $\alpha=4.67$}}
 \end{subfigure}
 \begin{subfigure}{.4\textwidth}
   \includegraphics[width=\textwidth]{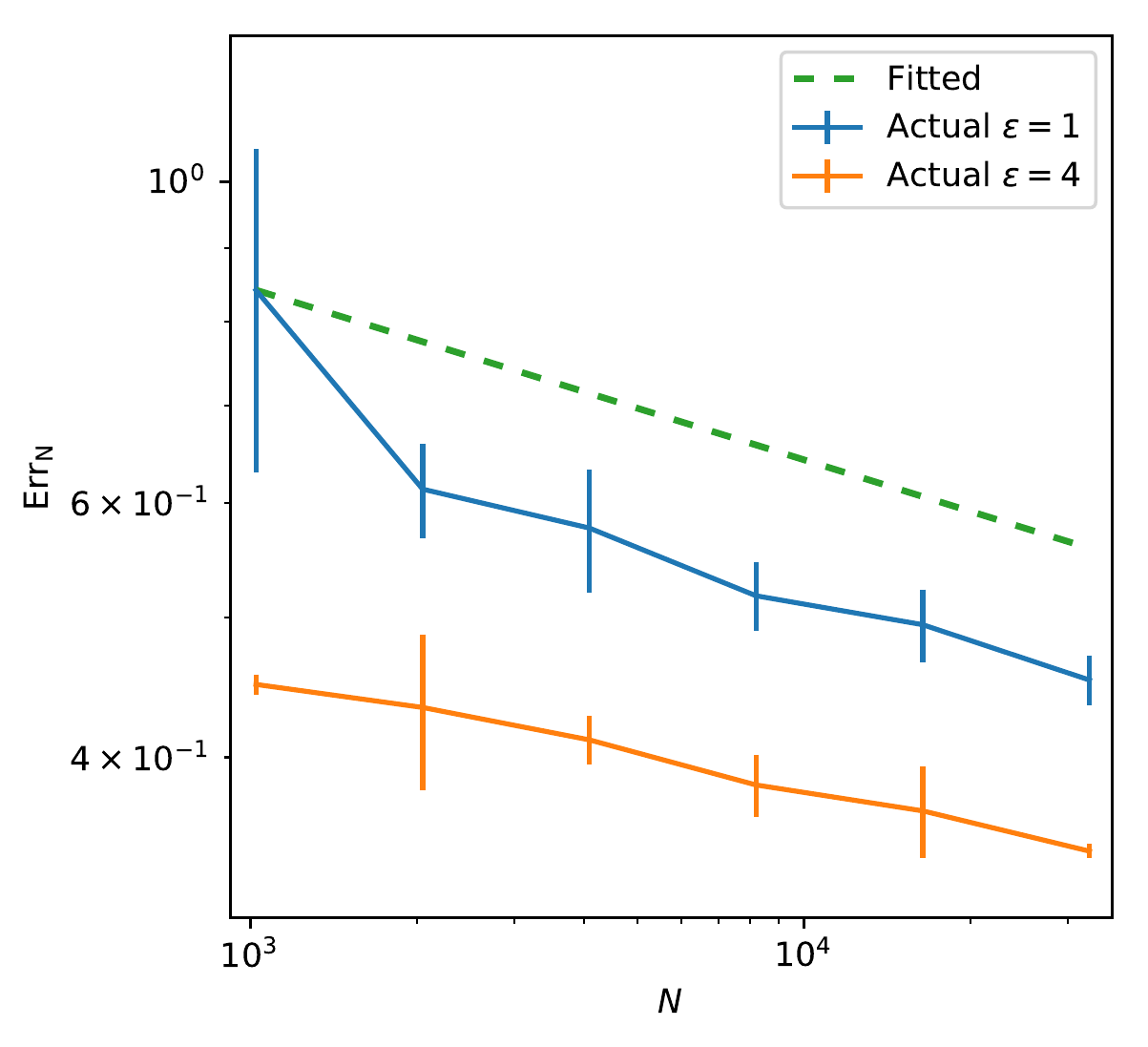}
   \caption{{furnishings, $\alpha=4.22$}}
 \end{subfigure}
 \begin{subfigure}{.4\textwidth}
   \includegraphics[width=\textwidth]{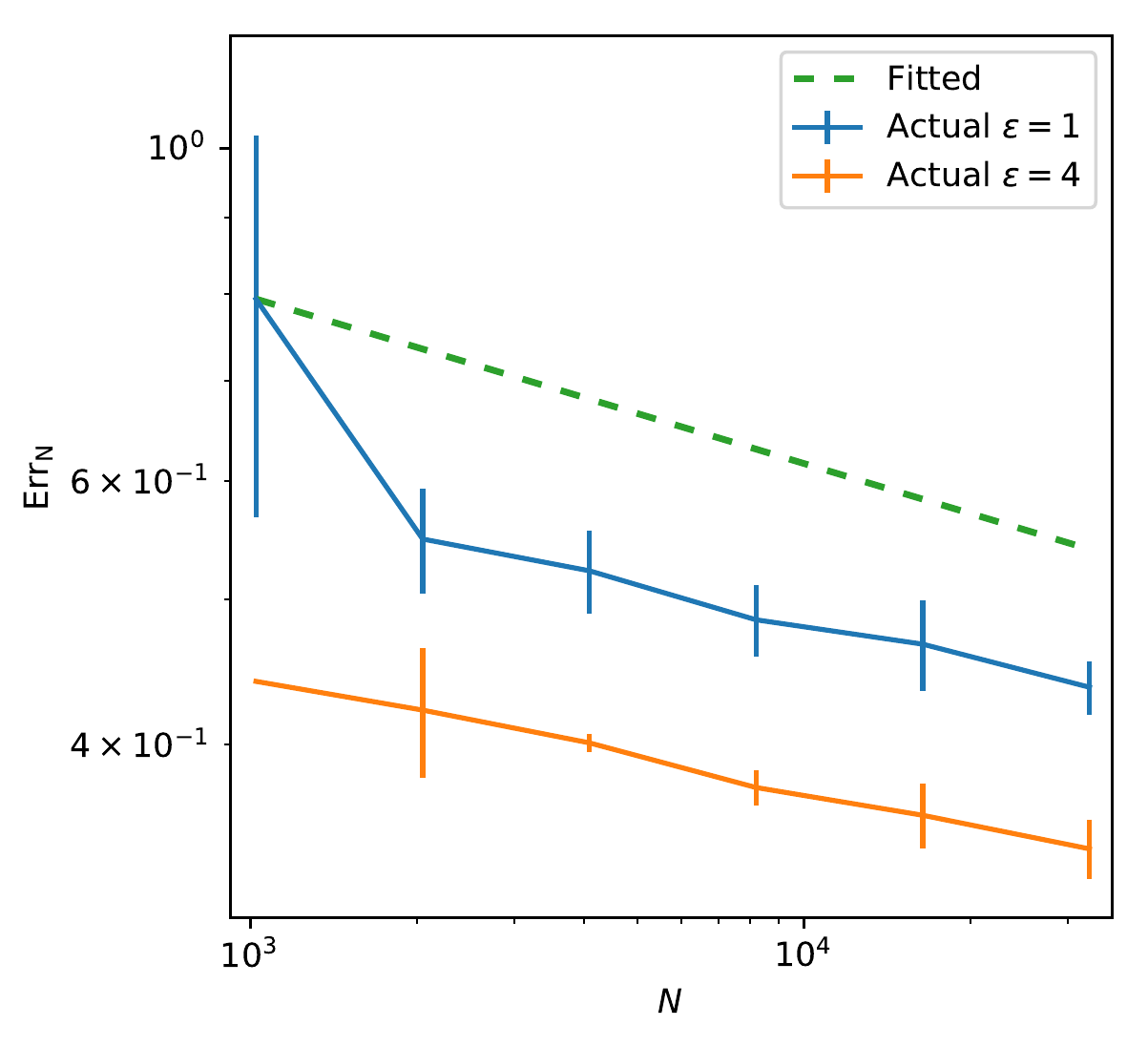}
   \caption{{game-and-toy, $\alpha=4.50$}}
 \end{subfigure}
 \caption{Experimental results for {\bfseries Task2}, Part II.}\label{fig:real-task22}
\end{figure}

\begin{figure}[H]
  \centering
 \begin{subfigure}{.4\textwidth}
   \includegraphics[width=\textwidth]{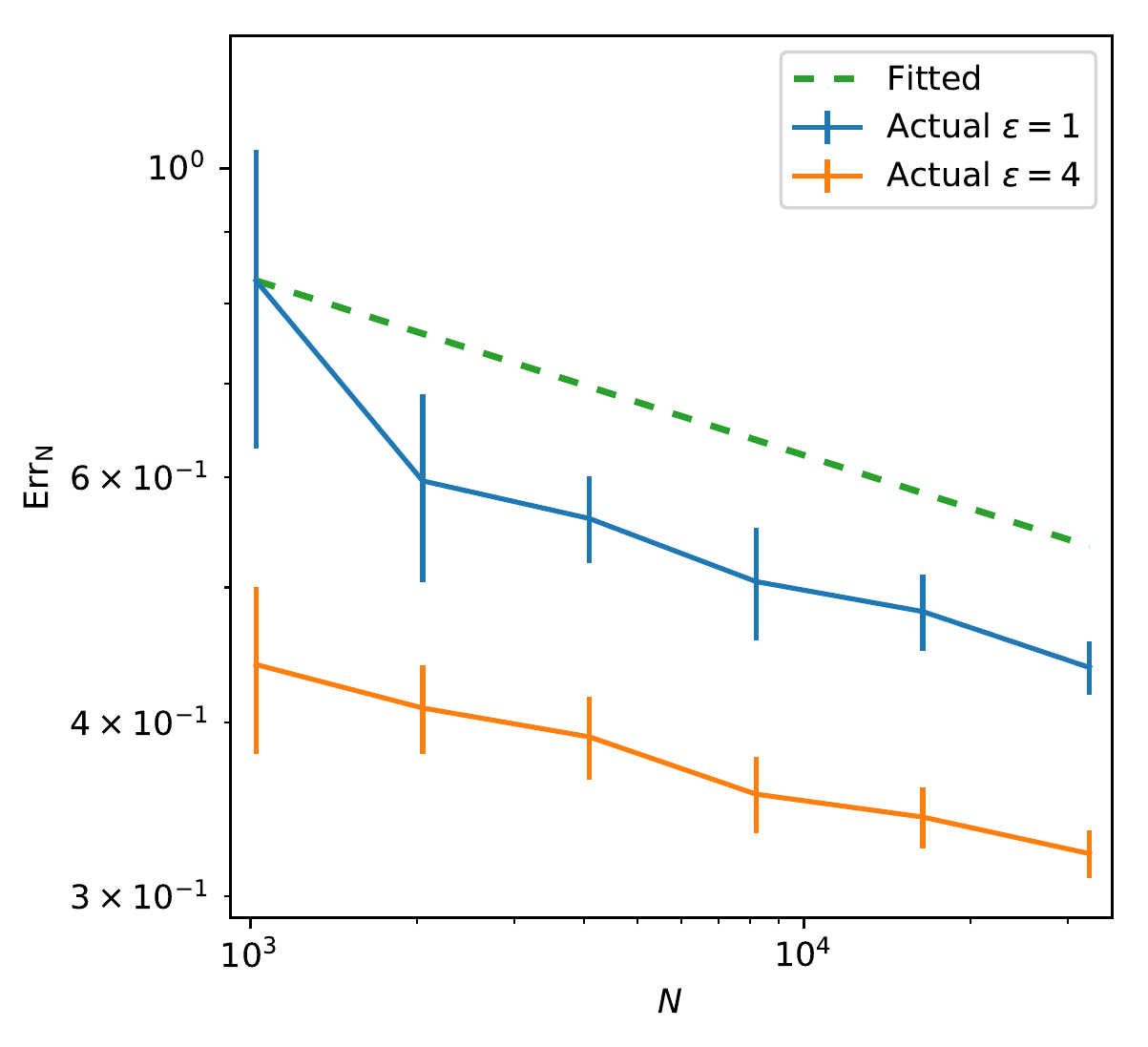}
   \caption{{instrument-hobby-learning, $\alpha=3.94$}}
 \end{subfigure}
 \begin{subfigure}{.4\textwidth}
   \includegraphics[width=\textwidth]{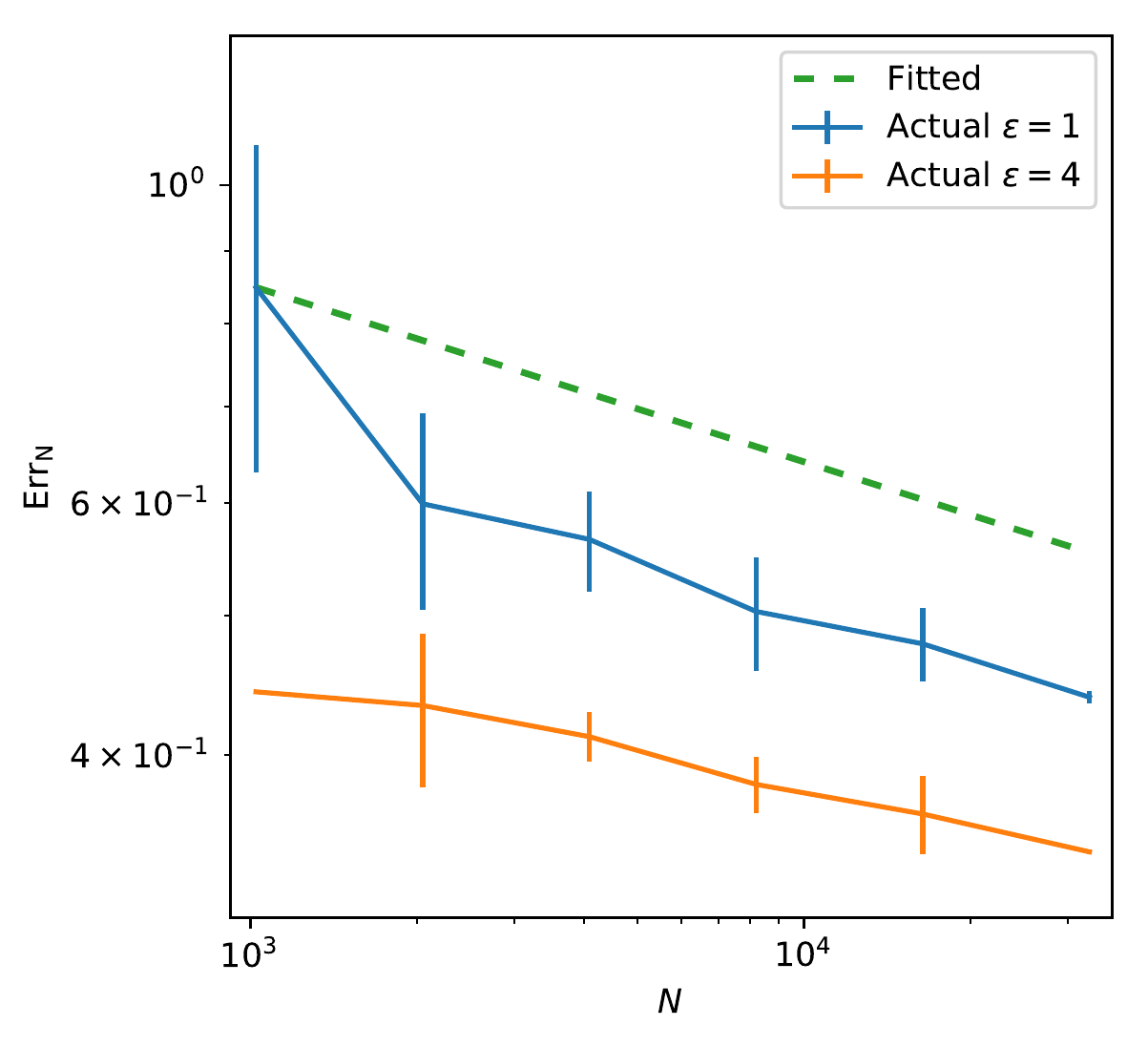}
   \caption{{kitchen-and-daily-goods, $\alpha=4.06$}}
 \end{subfigure}
 \begin{subfigure}{.4\textwidth}
   \includegraphics[width=\textwidth]{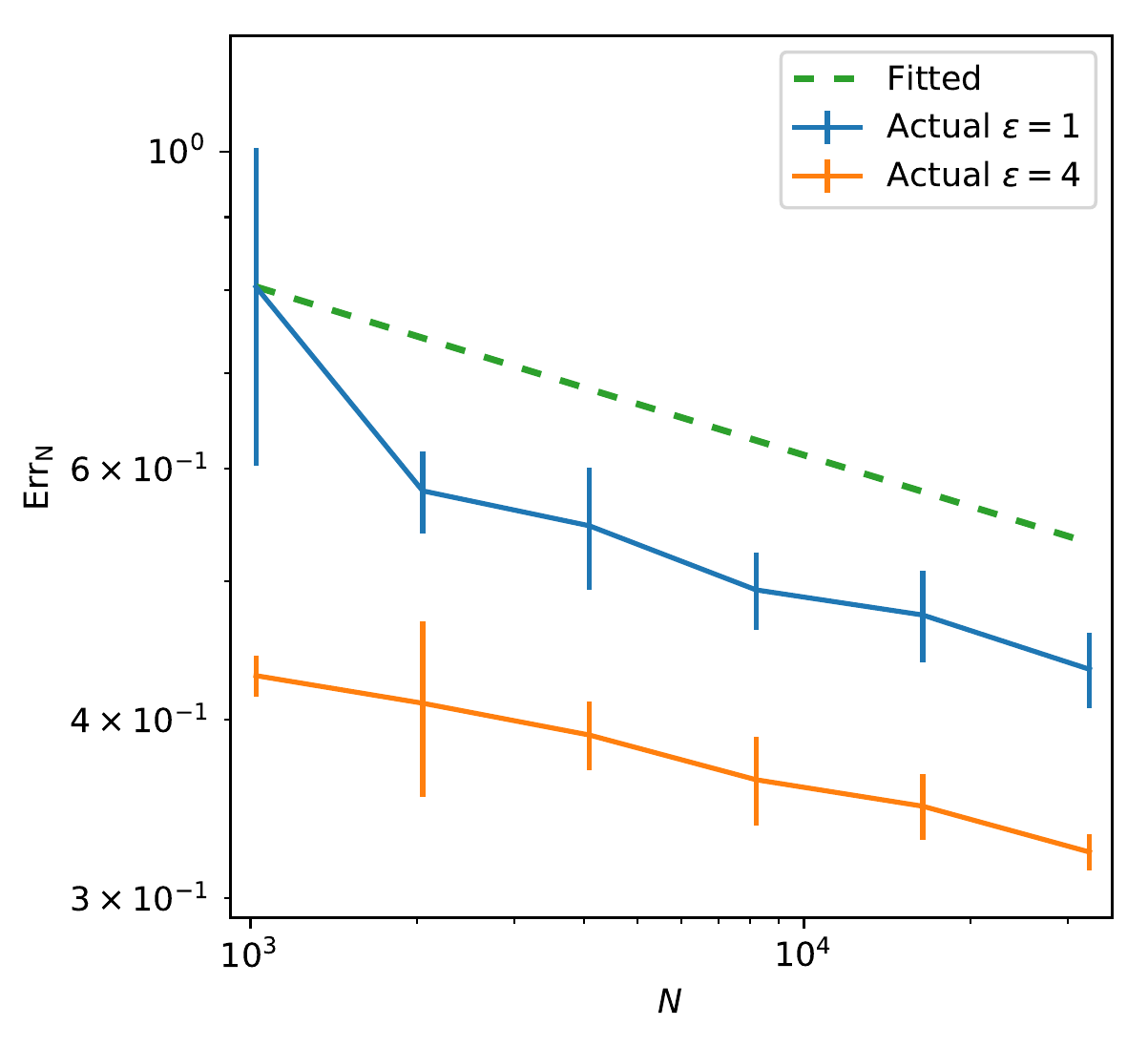}
   \caption{{leisure-and-outdoor, $\alpha=4.19$}}
 \end{subfigure}
 \begin{subfigure}{.4\textwidth}
   \includegraphics[width=\textwidth]{real/age-vs-category/filtered-result_music_software.pdf}
   \caption{{music-software, $\alpha=3.60$}}
 \end{subfigure}
 \begin{subfigure}{.4\textwidth}
   \includegraphics[width=\textwidth]{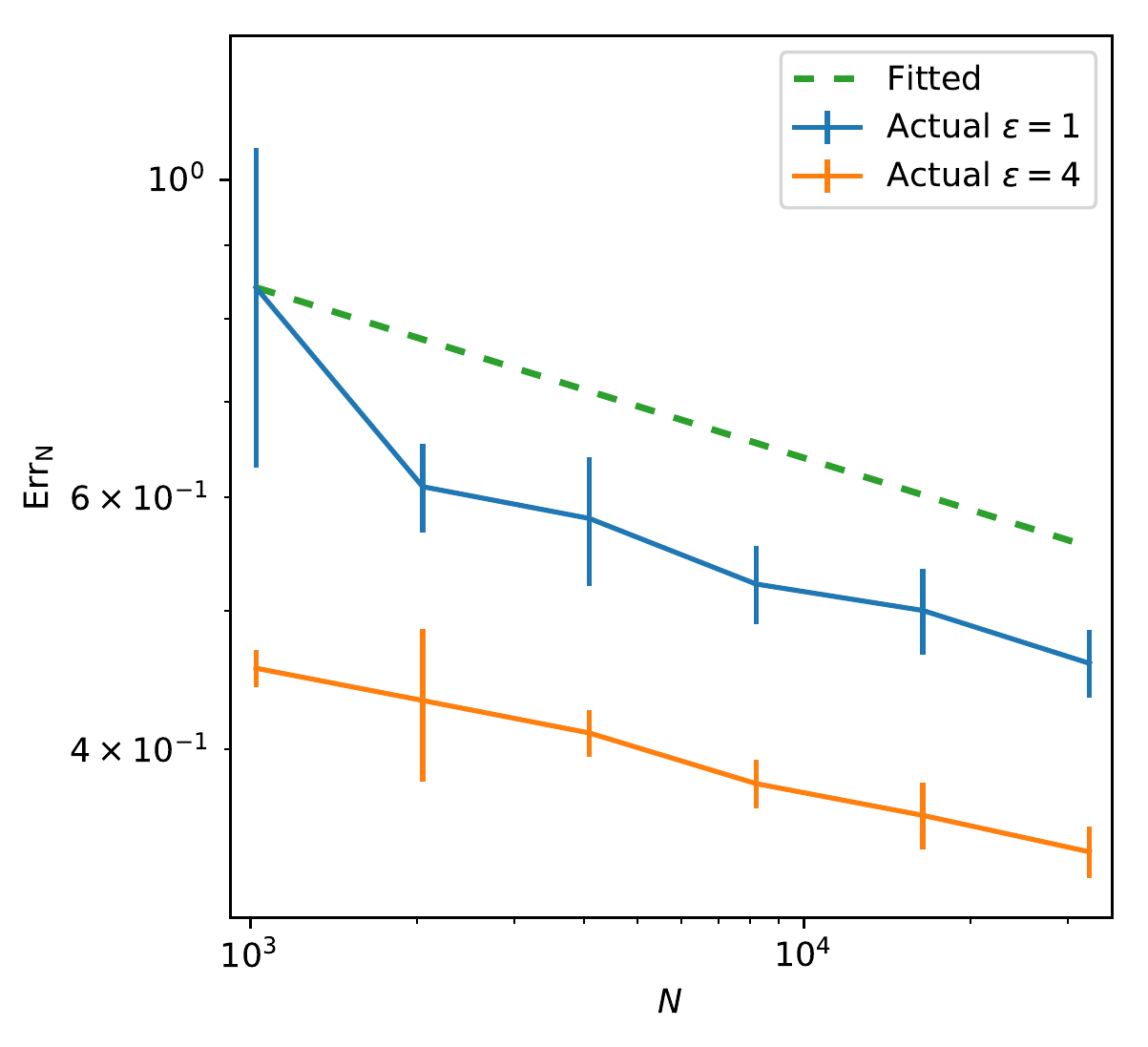}
   \caption{{personal-computer, $\alpha=4.15$}}
 \end{subfigure}
 \begin{subfigure}{.4\textwidth}
   \includegraphics[width=\textwidth]{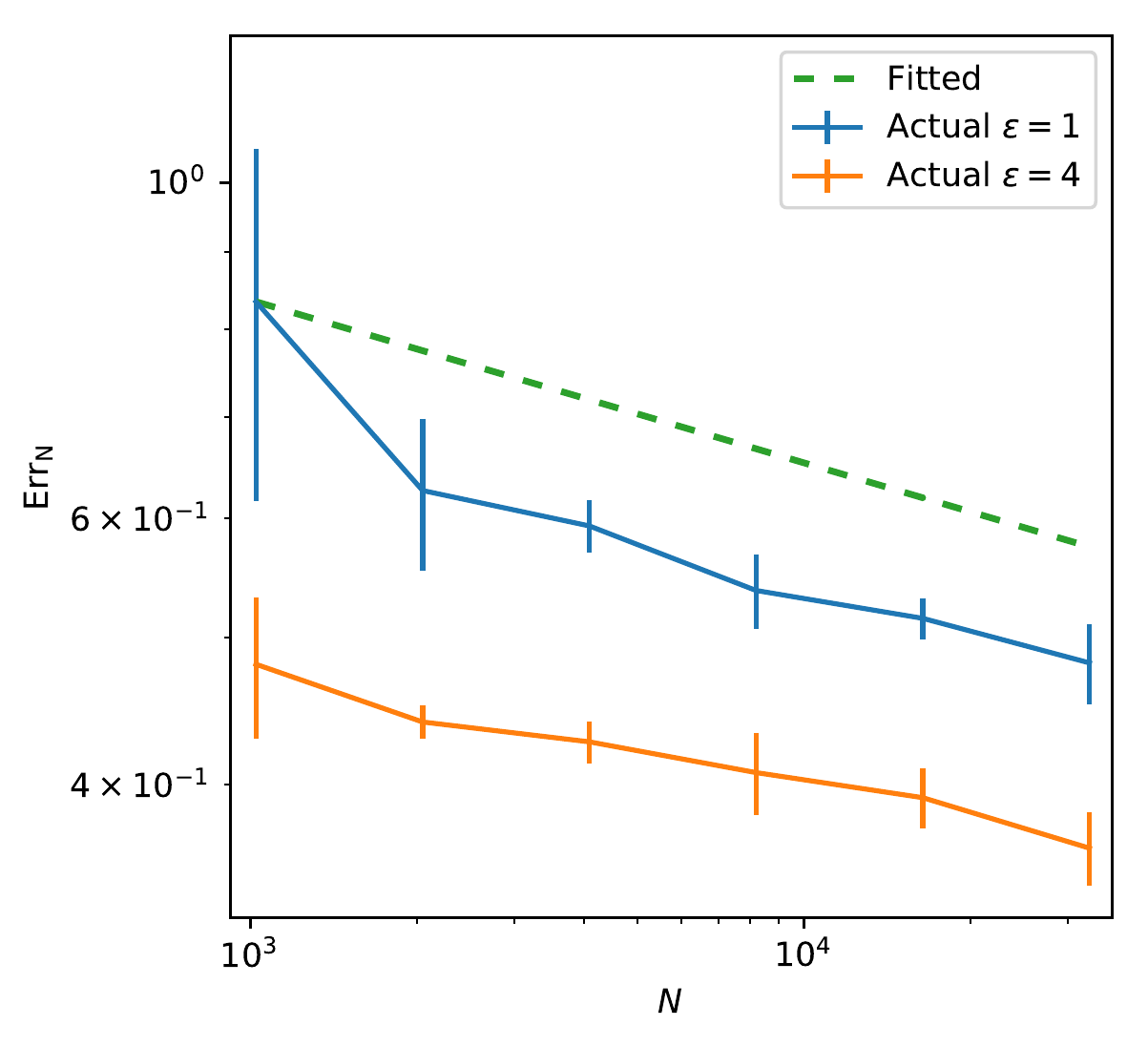}
   \caption{{pet-supplies, $\alpha=4.64$}}
 \end{subfigure}
 \caption{Experimental results for {\bfseries Task2}, Part III.}\label{fig:real-task23}
\end{figure}

\begin{figure}[H]
 \centering
 \begin{subfigure}{.4\textwidth}
   \includegraphics[width=\textwidth]{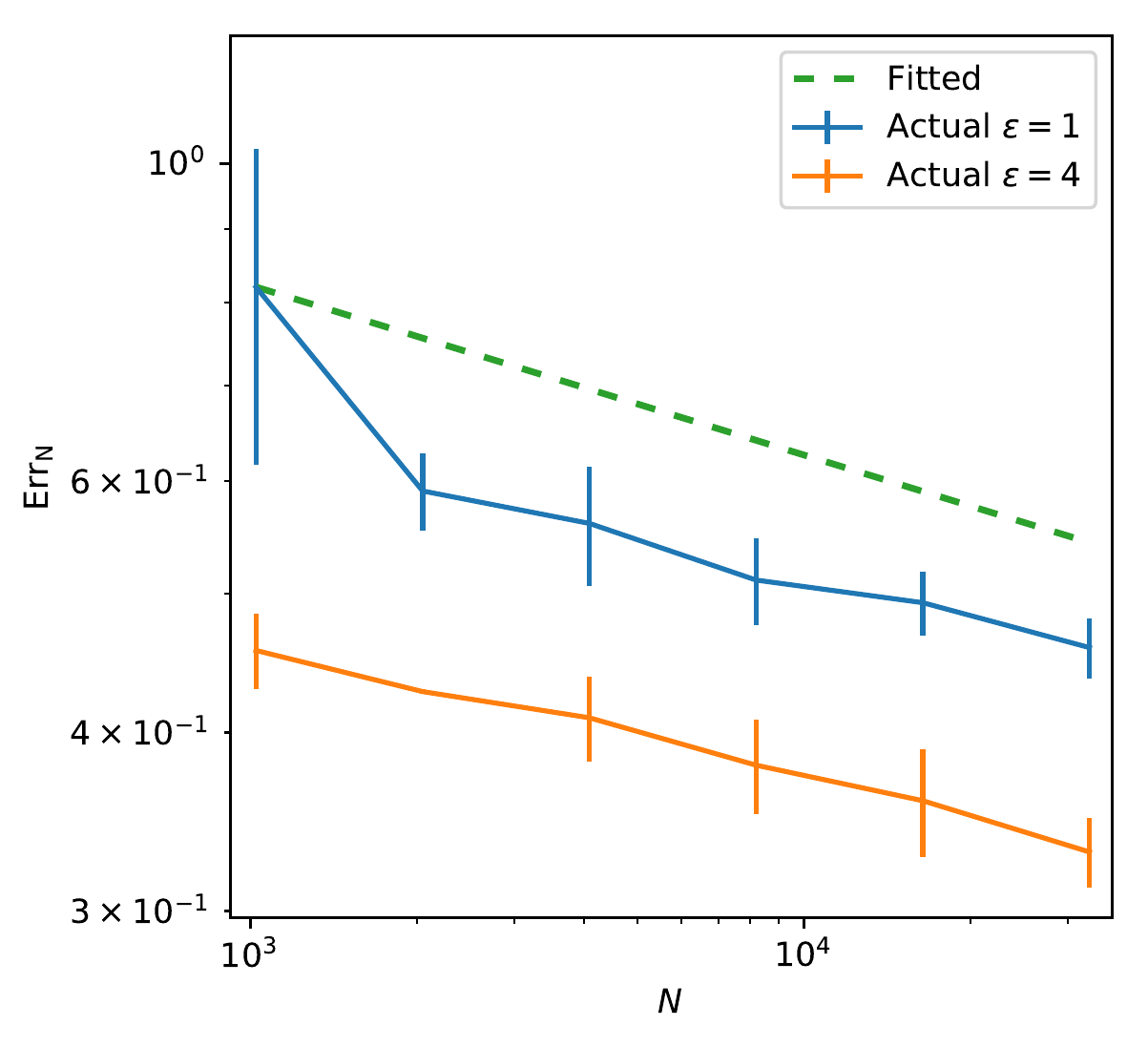}
   \caption{{sports, $\alpha=4.21$}}
 \end{subfigure}
 \caption{Experimental results for {\bfseries Task2}, Part IV.}\label{fig:real-task24}
\end{figure}

\end{document}